\documentclass{article}
\usepackage[utf8]{inputenc}
\usepackage[hmargin=33mm,vmargin=33mm]{geometry}

\usepackage[english]{babel}
\usepackage{graphicx}
 \usepackage{amsmath, amsthm, amssymb, mathabx, verbatim, setspace, enumerate,mathtools}
\usepackage[mathscr]{euscript}
\usepackage[all,cmtip]{xy}
\usepackage{etoolbox}
\usepackage{subfiles}
\usepackage{blindtext}
\usepackage{tikz-cd}
\usetikzlibrary{matrix, calc, arrows}

\usepackage{amsfonts,amsmath,amssymb,mathrsfs,amsthm}
\usepackage{verbatim}
\usepackage{mathtools}

\newcommand{\CC}{\mathbb{C}}
\newcommand{\ds}{\displaystyle}
\newcommand{\ra}{\rightarrow}
\newcommand{\ZZ}{\mathbb{Z}}
\newcommand{\RR}{\mathbb{R}}
\newcommand{\QQ}{\mathbb{Q}}

\newcommand{\TT}{\mathbb{T}}

\newcommand{\F}{\mathcal{F}}

\newcommand{\OO}{\mathcal{O}}

\newcommand{\HH}{\mathcal{H}}

\newcommand{\A}{\mathcal{A}}

\newcommand{\SSS}{\mathcal{S}}

\newcommand{\AAA}{\mathbb{A}}
\newcommand{\pp}{\mathfrak{p}}

\newtheorem{Thm}{Theorem}[section]
\newtheorem{Prop}[Thm]{Proposition}
\newtheorem{Lem}[Thm]{Lemma}
\newtheorem{cor}[Thm]{Corollary}

\theoremstyle{definition}

\newtheorem{rmk}[Thm]{Remark}

\DeclareMathOperator{\GL}{GL}

\DeclareMathOperator{\sgn}{sgn}

\DeclareMathOperator{\SL}{SL}
\DeclareMathOperator{\Gal}{Gal}
\DeclareMathOperator{\Nm}{Nm}

\DeclareMathOperator{\Frac}{Frac}
\DeclareMathOperator{\Char}{char}
\DeclareMathOperator{\ord}{ord}
\DeclareMathOperator{\Aut}{Aut}
\DeclareMathOperator{\vol}{vol}

\usepackage[OT2,T1]{fontenc}
\DeclareSymbolFont{cyrletters}{OT2}{wncyr}{m}{n}
\DeclareMathSymbol{\Sha}{\mathalpha}{cyrletters}{"58}

\DeclareMathOperator{\Cl}{Cl}
\DeclareMathOperator{\Tr}{Tr}

\begin{document}
\title{Central values of $L$-functions of cubic twists}
\author{Eugenia Rosu}
\date{}
\maketitle

\begin{abstract}
We are interested in finding for which positive integers $D$ we have rational solutions for the equation $x^3+y^3=D.$ The aim of this paper is to compute the value of the $L$-function $L(E_D, 1)$ for the elliptic curves $E_D: x^3+y^3=D$. For the case of $p$ prime $p\equiv 1\mod 9$, two formulas have been computed by Rodriguez-Villegas and Zagier in \cite{RV-Z}. We have computed formulas that relate $L(E_D, 1)$ to the square of a trace of a modular function at a CM point. This offers a criterion for when the integer $D$ is the sum of two rational cubes. Furthermore, when $L(E_D, 1)$ is nonzero we get a formula for the number of elements in the Tate-Shafarevich group and we show that this number is a square when $D$ is a norm in $\QQ[\sqrt{-3}]$.
\end{abstract}

\section{Introduction}
  
In the current paper we are interested in finding which positive integers $D$ can be written as the sum of two rational cubes:
\begin{equation}\label{main}
x^3+y^3=D, ~x, y\in \QQ.
\end{equation}

Despite the simplicity of the problem, an elementary approach to solving the Diophantine equation fails. However, we can restate the problem in the language of elliptic curves. 
After making the equation homogeneous, we get the equation $x^3+y^3=Dz^3$ that has a rational point at $\infty=[1:-1:0]$. Moreover, after a change of coordinates $\ds X=12D\frac{z}{x+y}$, $\ds Y=36D\frac{x-y}{x+y}$ the equation becomes:
\begin{equation*}
E_D: Y^2=X^3-432D^2,
\end{equation*}
which defines an elliptic curve over $\QQ$ written in its  Weierstrass affine form.
 
 Thus the problem reduces to finding if the group of rational points $E_D(\QQ)$ of the elliptic curve $E_D$ is non-trivial. We assume $D$ cube free and $D\neq1,2$ throughout the paper. In this case $E_D(\QQ)$ has trivial torsion (see \cite{S}), thus (\ref{main}) has a solution iff $E_D(\QQ)$ has positive rank. From the \textit{Birch and Swinnerton-Dyer(BSD)} conjecture, this is  conjecturally equivalent to the vanishing of $L(E_D, 1)$. 

Without assuming BSD, from the work of Coates-Wiles \cite{CW} (or more generally Gross-Zagier \cite{GZ} and Kolyvagin \cite{K}), when $L(E_D, 1)\neq 0$ the rank of $E_D(\QQ)$ is $0$, thus we have no rational solutions in $(\ref{main})$.

In the case of prime numbers, Sylvester conjectured that we have solutions in (\ref{main}) in the case of $D \equiv 4, 7, 8 \mod 9$. In the cases of $D$ prime with $D \equiv 2, 5\mod 9$, $D$ is not the sum of two cubes. This follows from a $3$-descent argument (given in the 19th century by Sylvester, Lucas and Pepin).

We define the invariant 
\[
S_D=\frac{L(E_D, 1)}{c_{3D} \Omega_{D}},
\]
where $\ds \Omega_{D}= \frac{\sqrt{3}}{6\pi\sqrt[3]{D}}\Gamma\left(\frac{1}{3}\right)^3$ is the real period and $c_{3D}=\prod\limits_{p|3D} c_p$ is the product of the Tamagawa numbers $c_p$ corresponding to the elliptic curve $E_D$ at the unramified places $p|3D$. The definition is made such that in the case of $L(E_D, 1)\neq 0$ we expect to get from the full BSD conjecture:
\begin{equation}
S_D=
\# \Sha(E_D) 
,\end{equation} where $\# \Sha$ is the order of the Tate-Shafarevich group.

\bigskip

From the work of Rubin \cite{R}, $L(E_D, 1)\neq 0$ implies the order of $\Sha(E_D)$ is finite. Furthermore, using the Cassels-Tate pairing, Cassels proved in \cite{C} that when $\Sha$ is finite the order $\#\Sha$ is a square. We actually show that, when $D$ is a norm in $\QQ[\sqrt{-3}]$, $S_D$ is an integer square up to an even power of $3$. Current work in Iwasawa theory shows that for semistable elliptic curves at the good primes $p$ we have $\ord_p(\#\Sha[p^{\infty}])=\ord_p(S_D)$, where $\Sha[p^{\infty}]$ is the $p^{\infty}$-torsion part of $\Sha$ (see \cite{JSW}). However, this cannot be applied at the place $3$ in our case.

   By computing the value of $S_D$, one can determine when we have solutions in \eqref{main} and, assuming the full BSD conjecture, one can find in certain cases the order of $\Sha$:
   \begin{enumerate}[(i)]
   	\item $S_D\neq 0$ $\Longrightarrow$ no solutions in \eqref{main}
	\item $S_D\neq 0$ $\xRightarrow{BSD}$ $S_D=\#\Sha$ integer square
	\item  $S_D=0$ $\xRightarrow{BSD}$ have solutions in \eqref{main}.
  \end{enumerate}

The goal of the current paper is to compute several formulas for $S_D$. In \cite{RV-Z}, Rodriguez-Villegas and Zagier computed formulas for $L(E_p, 1)$ in the case of primes $p\equiv 1\mod 9$. In the current paper we are extending on their results and compute similar formulas for all integers $D$.

Our main theorem is the following:

\begin{Thm}\label{thm3} For $D=\prod\limits_{p_i\equiv 1(3)} p_i^{e_i}$, $S_D$ is an integer square up to an even power of $3$.

\end{Thm}

Theorem \ref{thm3} above follows from the formula for $S_D$ presented below. Let $K=\QQ[\sqrt{-3}]$. For $D$ a norm in $\QQ[\sqrt{-3}]$, we write $D=D_1D_2^2$ such that $D_0=D_1D_2$ is the radical of $D$. Let $\pi_1, \pi_2 \equiv 1 \mod 3$ be elements of norm $D_1$ and $D_2$ respectively. Let $\sigma(D)$ the number of distinct primes dividing $D$ and $\varphi$ Euler's totient function.

\begin{Thm}\label{thm2} Using the above notation, let $D=\prod\limits_{p_i\equiv 1(3)} p_i^{e_i}$ be a positive integer that is a product of split primes in $K$ and $D_0=\prod_{p|D} p$ be its radical. Then $S_D$ is an integer square up to an even power of $3$ and we have:
  
\begin{equation}\label{two}
S_D=T_D^2 \frac{1}{(-3)^{2+\sigma(D)}} ,
\end{equation}



\noindent where the term $T_D/3$ is an integer if $\sigma(D)$ is even and $T_D/\sqrt{-3}$ is an integer if $\sigma(D)$ is odd. Moreover, we have the formula:
\[
T_D=\frac{1}{\varphi(D_0)}\Tr_{H_{\OO}/K}\left(\frac{\theta_1(\tau)}{\theta_0(\tau)}\omega^{k_0}\overline{\pi_1}^{-2/3} \pi_2^{1/3}\right),
\]
where:
\begin{itemize}
\item $\theta_r(z)=\sum\limits_{n\in\ZZ}(-1)^n e^{\pi i \left(n+\frac{r}{D}-\frac{1}{6}\right)^2 z}$, for $r=0, 1$ are theta functions of weight $1/2$,
\item $\tau=\frac{-b+\sqrt{-3}}{2}$ is a CM point such that $b^2\equiv -3\mod 12D^2$ and $(\pi_1\pi_2)^2$ divides $(\tau)$,
\item $H_{\OO}$ is the ray class field of modulus $3D_0$,
\item $\omega^{k_0}$ is the unique cube root of unity that makes $T_D/3$ or $T_D/\sqrt{-3}$ an integer.
\end{itemize}
\end{Thm}



This theorem follows from a more general result for all integers $D$ prime to $6$ that is proved using automorphic methods:

\begin{Thm}\label{thm1} For all integers $D$ prime to $6$, $3c_{3D} S_D$ is an integer and we have the formula:

\begin{equation}\label{one}
S_D= \frac{1}{3c_{3D}}\Tr_{H_{3D}/K} \left(D^{1/3}\frac{\Theta_K(D\omega)}{\Theta_K(\omega)}\right),
\end{equation}
where $\ds \Theta_K(z)=\sum\limits_{a, b\in \ZZ}e^{2\pi i z (a^2+b^2-ab)}$ is the theta function of weight one associated to the number field $K=\QQ[\sqrt{-3}]$, $\omega=\frac{-1+\sqrt{-3}}{2}$ is a third root of unity, and 
and $H_{3D}$ is the ring class field associated to the order $\OO_{3D}=\ZZ+3D\OO_{K}$.

\end{Thm}

Note that each of the elliptic curves $E_D$ is a cubic twist of $E_1$. In the case of quadratic twists of elliptic curves, an important tool in computing the values of the $L$-functions is the work of Waldspurger \cite{W}. For example, this is used to obtain Tunnell's theorem for congruent numbers in \cite{T}. However, the cubic twist case proves to be significantly more difficult. We instead take advantage of the fact that $E_D$ is an elliptic curve with complex multiplication by $\OO_K=\ZZ[\omega]$ the ring of integers of the number field $K=\QQ[\sqrt{-3}]$. Then from CM theory there is a Hecke character $\chi_{E_D}:\AAA_K^{\times}/K^{\times}\ra\CC^{\times}$ such that $L(E_D, s)=L_f(s, \chi_{E_D})$ and we compute the value of $L_f(s, \chi_{E_D})$ using automorphic methods.

\bigskip
 
 We present now an outline of the proof of Theorem \ref{thm1}. To compute the value of $L(s, \chi_{E_D})$ we look at the Hecke character adelically and using Tate's thesis we compute Tate's zeta function $Z(s, \chi_{E_D}, \Phi_K)$ for $\Phi_K$ a Schwartz-Bruhat function in $\SSS(\AAA_K)$. After integrating we get a linear combination of Eisenstein series. By evaluation at $s=1$, we write $L(E_D, 1)$ as a linear combination of theta functions at CM-points. We further show using Shimura's reciprocity law that the terms are all Galois conjugates over $K$.

The idea of the proof of Theorem \ref{thm2} is based on factoring each weight one theta function $\Theta_K(z)$ into a product of theta functions of weight $1/2$. The method we are using is a factorization lemma of Rodriguez-Villegas and Zagier from \cite{RV-Z2} applied to the formula in Theorem \ref{thm1}.
 This gives us the square of a linear combination of theta functions evaluated at CM points. Finally, using Shimura reciprocity law, we show that all the factors are Galois conjugates to each other and recover an integer square.

\bigskip
Note that using the formula \eqref{one} we can show that an integer $D$ cannot be written as the sum of two cubes by computationally checking whether $L(E_D, 1)\neq 0$. Furthermore, assuming BSD, $S_D=\#\Sha$ and thus we can compute the expected order of $\Sha$ explicitly.
\vspace{.2cm}

{\bf Acknowledgements.} The author would like to thank Xinyi Yuan for suggesting the problem and for valuable insights. We would also like to thank Don Zagier for encouragement to finish the result in the current form, very helpful discussions and for help with using PARI for computational purposes. We would also like to thank Max Planck Institute in Bonn and Tsinghua University in Beijing for their hospitality.

\bigskip



\section{Background}

Let $K=\QQ[\sqrt{-3}]$ and denote $\OO_K=\ZZ[\omega]$ its ring of integers, where $\omega=\frac{-1+\sqrt{-3}}{2}$ is a fixed cube root of unity. We will denote by $K_v$ the localization of $K$ at the place $v$, and $K_p=\prod\limits_{v|p} K_v  \cong \QQ_p[\sqrt{-3}]$. Note that $\ZZ_p[\omega]\cong \prod\limits_{v|p} \OO_{K_v}$.

\subsection{The $L$-function}

 Our goal is to compute several formulas for the central value of the $L$-function $L(E_D, 1)$ of the elliptic curve $E_D: x^3+y^3=Dz^3$. The elliptic curve $E_D$ has complex multiplication (CM) by $\OO_K$. Then from CM theory we can find a Hecke character $\chi:\AAA_{K}^{\times}/K^{\times} \ra \CC$ corresponding to the elliptic curve $E_D$ such that $L(E_D, s)=L_f(s, \chi_D\varphi)$. We can compute explicitly $\chi=\chi_D\varphi$ (see Ireland and Rosen \cite{IR} for more details), where $\varphi$ is the Hecke character corresponding to $E_1$ and $\chi_D$ is the Hecke character corresponding to the cubic twist. More precisely, writing the characters classically, we have:
	\begin{itemize}
		\item $\varphi: I(3) \ra K^{\times}$ is defined on the set of ideals prime to $3$ by taking $\varphi(\A)=\alpha$,  where $\alpha$ is the unique generator of the ideal $\A$ such that  $\alpha\equiv 1 \mod 3$.
		
		\item $\chi_D: \Cl(\OO_{3D}) \ra \{1, \omega, \omega^2\}$ is the cubic character defined below in Section \ref{cubic}; it is defined over $\Cl(\OO_{3D})$ the ring class group corresponding to the order $\OO_{3D}=\ZZ+3D\OO_K$.
		
	\end{itemize}

 Note that the $L$-function can be expanded as $\ds L(E_D, s)=\sum_{\substack{\alpha\in \OO_K \\ \alpha\equiv 1 (\text{mod} ~3)}}  \frac{\chi_D(\alpha) \alpha}{\Nm\alpha^s}.$






\subsection{The cubic character}\label{cubic}

We define the cubic character $\chi_D$ and recall some of its properties following Ireland and Rosen \cite{IR}. Let $\omega=\frac{-1+\sqrt{-3}}{2}$ and for $\alpha \in \ZZ[\omega]$ prime to $3$, we define the cubic residue character $\left(\frac{\alpha}{\cdot}\right)_3:I(3\alpha)\ra \{1, \omega, \omega^2\}$, where $I(3\alpha)$ is the set of fractional ideals of $K$ prime to $3\alpha$. For a prime ideal $\mathfrak{p}$ of $K$, we define $\left(\frac{\alpha}{\pp}\right)_3=\omega^j$, for  $0\leq j\leq 2$ such that 
\[
\alpha^{(\Nm\mathfrak{p}-1)/3}\equiv \omega^j \mod \mathfrak{p}.
\]
It is further defined multiplicatively on the fractional ideals of $I(3\alpha)$. 

It is easy to check that the definition makes sense, as the group $(\ZZ[\omega]/\pp\ZZ[\omega])^{\times}$ has $\Nm \pp -1$ elements, thus $\alpha^{\Nm \pp -1} \equiv 1\mod \pp$. As $\Nm \pp \equiv 1\mod 3$, we can factor out $\ds \alpha^{\Nm \pp -1} -1 =(\alpha^{(\Nm \pp -1)/3}-1)(\alpha^{(\Nm \pp -1)/3}-\omega)(\alpha^{(\Nm \pp -1)/3}-\omega^2)$ and as $K$ is an UFD, $\pp$ divides exactly one of these terms, exactly $\alpha^{(\Nm \pp -1)/3}-\left(\frac{\alpha}{\pp}\right)_3$.

The character $\chi_D$ is defined following \cite{IR} to be:
\[
\chi_{D}(\A)=\overline{\left(\frac{D}{\A}\right)_3}.
\]
We also define $\chi_{\pi}(\A)=\overline{\left(\frac{\pi}{\A}\right)_3}$ where $\pi$ is a generator of an ideal of norm $D$. Note that $\overline{\chi_{\pi}(\A)}=\chi_{\overline{\pi}}(\overline{\A})$.

An important result is the cubic reciprocity law (see \cite{IR} for more details):

\begin{Thm}
\textbf{(Cubic reciprocity law)}. For $\pi_1, \pi_2 \equiv 2 \mod 3$ generators of the prime ideals $\mathfrak{p}_1, \mathfrak{p}_2$ prime to $3$, we have $\ds \left(\frac{\pi_1}{\pi_2}\right)_3=\left(\frac{\pi_2}{\pi_1}\right)_3.$

\end{Thm}

It follows immediately from the cubic reciprocity law that for $\alpha\equiv \pm 1\mod 3$, we have  $\chi_D(\alpha)=\chi_{\alpha}(D)$. Also from the cubic reciprocity it follows that $\chi_D((\alpha))=1$ for $\alpha \equiv a \mod 3D$, where $a$ is an integer prime to $3D$. Thus $\chi_D$ is invariant on the ideals of $P_{\ZZ, 3D}=\{ (\alpha)$: $\alpha \in K$ such that $\alpha \equiv a \mod 3D$ for some integer $a$ such that $(a, 3D)=1$\}. The ring class group of the order $\OO_{3D}=\ZZ+3D\OO_K$ is defined to be $\Cl(\OO_{3D})=I(3D)/P_{\ZZ, 3D}$, where $I(3D)$ is the set of fractional ideals prime to $3D$, and thus 
$\chi_D$ is invariant on $\Cl(\OO_{3D})$.


\bigskip
Finally, we can relate the cubic character to the Galois conjugates of $D^{1/3}$:

\begin{Lem}\label{GaloisD}
Let $D$ be an integer prime to $3$ and $\pi$ a generator of an ideal of norm $D$. Then for an ideal $\A$ of $K$ prime to $3D$, we have:
\[
\pi^{1/3}\chi_{\pi}(\A)=(\pi^{1/3})^{\sigma_{\A}^{-1}},
\]
 where $\sigma_{\A} \in \Gal(\CC/K)$ is the Galois action corresponding to the ideal $\A$ in the Artin correspondence. Note that this immediately implies $D^{1/3}\chi_{D}(\A)=(D^{1/3})^{\sigma_{\A}^{-1}}.$
  \end{Lem}

\begin{proof} It is enough to show the result for a prime ideal $\mathfrak{p}$ of $K$, $\pp$ prime to $3D$. Let $\sigma_{\mathfrak{p}}=\left(\frac{L/K}{\frak{p}}\right)$ be the Frobenius element  corresponding to the prime ideal $\mathfrak{p}$ of $\OO_K$, where $L=K[\pi^{1/3}]$. Then from the definition of the Frobenius, element for $\pi^{1/3} \in L$, we get $(\pi^{1/3})^{\sigma_{\mathfrak{p}}}\equiv (\pi^{1/3})^{\Nm\mathfrak{p}} \mod ~\mathfrak{p}\OO_L$.

Furthermore, note that  $(\pi^{1/3})^{\Nm\mathfrak{p}}= \pi^{1/3}\pi^{(\Nm\mathfrak{p}-1)/3} \equiv \pi^{1/3} \overline{\chi_{\pi}(\mathfrak{p})} \mod \mathfrak{p}\OO_L$. Since the Galois conjugates of $\pi^{1/3}$ are the roots of $x^3-\pi$, the Galois conjugates of $\pi^{1/3}$ must be $(\pi^{1/3})^{\sigma_{\pp}}\in \{\pi^{1/3}, \pi^{1/3}\omega, \pi^{1/3}\omega^2\}$, and from the congruences above we get $(\pi^{1/3})^{\sigma_{\pp}}=\pi^{1/3} \overline{\chi_{\pi}(\mathfrak{p})}.$ Changing $\pp$ to $\pp^{-1}$ we get the result. \end{proof}

\subsection{Hecke characters} 

A classical Hecke character $\widetilde{\chi}: I(f)\ra \CC^{\times}$ of conductor $f$ can be expressed on the set of principal ideals $P(f)$ prime to $f$ in the form $\ds \widetilde{\chi}((\alpha))=\widetilde{\epsilon}(\alpha)\widetilde{\chi}^{-1}_{0}(\alpha),$ where $\widetilde{\varepsilon}: (\OO_K/f\OO_K)^{\times} \ra \mathbb{T}$ is a character taking values in a finite group $\TT$ and $\widetilde{\chi}_{0}$ is an infinity type continuous character, meaning that $\widetilde{\chi}_{0}:\CC^{\times}\ra \CC^{\times}$ is a continuous character. 

The idelic Hecke character is a continuous character $\chi: \AAA^{\times}/K^{\times} \ra \CC^{\times}$. There is a unique correspondence between the idelic and the classical Hecke characters defined as follows: at $\infty$ for $z\in \CC$ we define $\chi_{\infty}(z)={\widetilde{\chi}_0}^{-1}(z)$ for $z\in \CC^{\times}$ and at the places $v\nmid f$ we define $\chi(\OO_v^{\times}\varpi_v):=\widetilde{\chi}(\mathfrak{p}_v)$, for $\varpi_v$ a uniformizer of $\OO_{K_v}$ and $\pp_v$ the prime corresponding to the place $v$. At the places $v|f$, the value of $\chi_v$ can be determined using the Weak Approximation Theorem.

We are interested in the character $\chi=\chi_{D}\varphi$ defined before. By abuse of notation, we will use $\varphi, \chi_{D}$ both for the classical and the adelic Hecke characters.

Recall $\varphi:I(3)\ra \CC^{\times}$ is the Hecke character defined by $\chi((\alpha))=\alpha$ for $\alpha \equiv 1 \mod 3$. For the place $v\nmid 3$, denote by $\varpi_v$ a uniformizer of $\OO_{K_v}$ such that $\varpi_v\equiv 1\mod 3$. Then for $\varphi: \AAA_{K}^{\times} \ra \CC^{\times}$, we can define:

\begin{itemize}
\item $\varphi_v(p)=-p$,  $\varphi_v(\OO_{K_v}^{\times})=1$, for $v=p,  p\equiv 2 \mod 3$,  
\item $\varphi_v(\varpi_v)=\varpi_v$, $\varphi_v(\OO_{K_v}^{\times})=1$, for  $v|p,  p \equiv 1 \mod 3$, 

\item $\varphi_{\infty}(x_{\infty})=x_{\infty}^{-1}$,  at $v=\infty$.
\end{itemize}

Recall $\chi_{D}:I(3D)\ra \{1, \omega, \omega^2\}$ is the cubic character and we showed that it is well-defined on $\Cl(\OO_{3D})$, the ring class group for the order $\OO_{3D}=\ZZ+3D\OO_K$. We define the character $\chi_D$ adelically over $K^{\times}\setminus \AAA_{K, f}^{\times}/U(3D)\simeq \Cl(\OO_{3D})$, where $U(3D)=(1+3\ZZ_3[\omega])\prod\limits_{p|D}{(\ZZ+D\ZZ_p[\omega])^{\times}}\prod\limits_{p\nmid 3D}{(\ZZ_p[\omega])}^{\times}$.

Note that we can rewrite $l_f\in \AAA_{K, f}^{\times}$ in the form $l_f=kl_1$ with $k\in K^{\times}$ and $l_1\in \prod\limits_{v\nmid\infty} \OO_{K_v}^{\times}$. We can find $k_1\in \OO_K$ such that $k_1 \equiv l_1 \mod 3D\OO_{K_v}$ and we define $\chi_{D, f}(l)=\chi_{D, f}(l_1)=\chi_{D}((k_1)).$ More precisely, we get:

\begin{itemize}
\item $\chi_{D, v}(\varpi_v)=\chi_{D}(\pp_v)$ and $\chi_{D, v}(\OO_{K_v}^{\times})=1$,  if  $v\nmid 3D$,
\item $\chi_{D,\infty}(x_{\infty})=1$,  at $v=\infty$.
\end{itemize}

The values of $\chi_D$ and $\varphi$ at the ramified places can be computed using the Weak approximation theorem. 


\section{$L(E_D, 1)$ and Tate's zeta function}\label{Zeta_functions}

In this section we will compute the value of $L(E_D, 1)=L(1, \chi_D\varphi)$, working with $\chi_D, \varphi$ as automorphic Hecke characters. Let $K=\QQ[\sqrt{-3}]$ and $\omega=\frac{-1+\sqrt{-3}}{2}$ a fixed cube root of unity as before. We will show the following result:

\begin{Thm}\label{thm_1_theta} For $\ds
S_D=\frac{2\sqrt{3}\pi D^{-1/3}}{c_{3D} \Gamma\left(\frac{1}{3}\right)^3}L(E_D, 1)
$, we have $3c_{3D} S_D \in \ZZ$ and
\begin{equation}\label{S_D_theta}
S_D= \frac{1}{3c_{3D}}\Tr_{H_{3D}/K} \left(D^{1/3}\frac{\Theta_K(D\omega)}{\Theta_K(\omega)}\right),
\end{equation}
where $\Theta_K(z)=\sum\limits_{m, n\in \ZZ} e^{2\pi i (m^2+n^2-mn)z}$,  $H_{3D}$ is the ring class field for the order $\OO_{3D}=\ZZ+3D\OO_{K}$ and $c_{3D}=\prod\limits_{p|3D} c_p$ is the product of the Tamagawa numbers $c_p$ of $E_D$.
\end{Thm}

We will compute the formula (\ref{S_D_theta}) using Tate's zeta function. We start by recalling some background and notation.

\subsection{Haar measure}
We take $V=K$ as a quadratic vector space over $\QQ$ with the norm as its quadratic form. We take $dx_v$ to be the the self-dual additive Haar measure and $d^{\times} \alpha_v$ the multiplicative Haar measure $d^{\times}_vx_v=\frac{dx_v}{|x_v|_v}$ normalized such that $\vol(\OO_{K_v}^{\times})=1$, if $v\nmid \infty$, and $d^{\times}z =\frac{dz}{|z|_{\infty}}$  where $dz$ is the usual Lebesgue measure, and $|z|_{\infty}=|z|_{\CC}^2$ is the square of the usual absolute value over $\CC$.

\subsection{Schwartz-Bruhat functions}
We choose the Schwartz-Bruhat function $\Phi_f\in S(\AAA_{K, f})$ such that Tate's zeta function $Z(s, \Phi, \chi_D\varphi)$ defined below to be nonzero. More precisely, $\Phi_f=\prod\limits_{v\nmid \infty} \Phi_v$, where:

\begin{itemize}
\item $\Phi_v=\Char_{\OO_{K_v}}$ for $v\nmid 3D$,
\item $\Phi_p=\sum\limits_{(a, D)=1}\Char_{(a+D\ZZ_p[\omega])}$ for $p|D$,
\item $\Phi_v=\Char_{(1+3\OO_{K_v})}$ for $v=\sqrt{-3}$.
\end{itemize}

\subsection{Tate's zeta function}

We recall Tate's zeta function. For a Hecke character $\chi:\AAA_{K}^{\times}/K^{\times} \ra \CC^{\times}$ and a Schwartz-Bruhat function $\Phi\in \SSS(\AAA_{K})$, Tate's zeta function is defined locally as $\ds Z_v(s, \chi_v, \Phi_v)=\int\limits_{K_v^{\times}} \chi_{v}(\alpha_v) |\alpha_v|^s_v \Phi_v(\alpha_v) d^{\times} \alpha_v$, and globally as $Z(s, \chi, \Phi)=\prod\limits_v Z_v(s, \chi_v, \Phi_v)$. As a global integral this is 
\[
\ds Z(s, \chi, \Phi)
=
\int\limits_{\AAA_{K}^{\times}} \chi(\alpha) |\alpha|^s \Phi(\alpha) d^{\times} \alpha.
\]
It has meromorphic continuation to all $s\in \CC$ and in our case it is entire. We will compute $Z_f(s, \chi_f, \Phi_f)$ for $\chi=\chi_D\varphi$ and the Schwartz-Bruhat function $\Phi_f$ chosen above.

From Tate's thesis, we have the equality of local factors $L_v(s, \chi_D\varphi)=Z_v(s, \chi_D\varphi)$ at all the unramified places, and thus $\ds L_f(s, \chi_D\varphi)=Z_f(s, \chi_D\varphi)\prod\limits_{p|3D}\frac{L_p(s, \chi_{D, p}\varphi_p)}{Z_p(s, \chi_{D, p}\varphi_p, \Phi_p)}$. As $\varphi, \chi_D$ and $|\cdot|$ are trivial when $\Phi_{p}$ is nonzero for $p|3D$, we can compute easily $\prod\limits_{p|3D}Z_p(s, \chi_{D, p}\varphi_p, \Phi_p)=\prod\limits_{p|D}\vol\left(\ZZ+3D\ZZ_p[\omega]\right)^{\times} \vol\left(1+3\ZZ_3[\omega]\right)^{\times}$ and this equals $\frac{1}{6}\prod_{p|D}(p-\left(\frac{p}{3}\right))^{-1}$. The terms $L_p(s, \chi_D\varphi)=1$ for $p|3D$ by definition. Thus for all $s$ and for $\Phi$ the Schwartz-Bruhat function chosen above, we have:
\begin{equation}\label{L_to_Z0} 
L_f(s, \chi_D\varphi)=Z_f(s, \chi_D\varphi, \Phi)V_{3D},
\end{equation}
\noindent where $\ds V_{3D}=\frac{1}{6}\prod_{p|D}(p-\left(\frac{p}{3}\right))^{-1}$. 


Next we compute the value of $Z_f(s, \chi_{D, f}\varphi_f, \Phi_f)$ as a linear combination of Hecke characters and use (\ref{L_to_Z0}) to get the value of $L_f(s, \chi_{D, f}\varphi_f)$:

\begin{Lem}\label{Z_linear} For all $s\in \CC$ and the Schwartz-Bruhat function $\Phi_f \in \SSS(\AAA_{K, f})$ chosen above, we have:
\[
L_f(s, \chi_D\varphi)=\sum_{\alpha_f \in U(3D)\setminus\AAA_{K, f}^{\times}/K^{\times}} I(s, \alpha_f, \Phi_f)\chi_D(\alpha) \varphi(\alpha),
\]
\noindent where $\ds I(s, \alpha_f, \Phi_f)=\sum\limits_{k\in K^{\times}} \frac{k}{|k|_{\CC}^{2s}} \Phi_f(k\alpha_f)$ and $U(3D)=(1+3\ZZ_3[\omega])\prod\limits_{v|D}{(\ZZ+D\ZZ_p[\omega])^{\times}}\prod\limits_{v\nmid 3D}{\OO_{K_v}^{\times}}$. 

\end{Lem}

\begin{proof} We first take the quotient by $K^{\times}$ in the integral defining $Z_f(s, \chi_D\varphi, \Phi_f)$ and get:
\[
Z_f(s, \chi_D\varphi, \Phi_f)=\int\limits_{\AAA_{K, f}^{\times}/K^{\times}}\sum_{k\in K^{\times}} \chi_{D, f}(k\alpha'_f)\varphi_f(k\alpha'_f) |k\alpha_f|_f^s \Phi_f(k\alpha'_f) d^{\times}\alpha'_f.
\]
We have  $\chi_{D, f}(k\alpha'_f)=\chi_{D, \infty}^{-1}(k)\chi_{D, f}(\alpha'_f)=\chi_{D, f}(\alpha'_f),$ $\varphi_f(k\alpha'_f)=\varphi_{\infty}^{-1}(k)\varphi_f(\alpha'_f)=k\varphi_f(\alpha'_f)$
and $|k\alpha'_f|^s_f=|k|^{-s}_{\infty}|\alpha_f|^s_f=|k|_{\CC}^{-2s}|\alpha_f'|_f^s$, where $| \cdot |_{\CC}$ is the usual absolute value over $\CC$.
Then the integral reduces to:
\[
Z_f(s, \chi_D\varphi, \Phi_f)
=
\int\limits_{\AAA_{K, f}^{\times}/K^{\times}}\left(\sum_{k\in K^{\times}} \frac{k}{|k|_{\CC}^{2s}}\chi_{D, f}(\alpha'_f) \Phi_f(k\alpha_f')\right) \varphi_f(\alpha'_f) |\alpha'_f|_f^s ~d^{\times} \alpha'_f.
\]
Furthermore, our Schwartz-Bruhat functions $\Phi_f(k\alpha_f')$ are invariant on $U(3D)$. Similarly, $|\cdot|_f$ is trivial on units, thus on $U(3D)$, while $\chi_D$ is invariant on $U(3D)$ by definition. Moreover, $\varphi$ is trivial on all the units at all the unramified places, while, at $3$, $\varphi$ is invariant under $1+3\ZZ_3[\omega]$, thus it is trivial on all of $U(3D)$. Thus we can take the quotient by $U(3D)$ as well. Note that the integral is now a finite sum:
\[
Z_f(s, \chi_D\varphi, \Phi_f)
=
\vol(U(3D))
\sum\limits_{\alpha''_f \in U(3D)\setminus \AAA_{K, f}^{\times}/K^{\times}}\left(\sum_{k\in K^{\times}} \frac{k}{|k|_{\CC}^{2s}}\Phi_f(k\alpha_f'')\right) \chi_{D, f}(\alpha''_f) \varphi_f(\alpha''_f) |\alpha''_f|_f^s.
\]
 We compute $\vol(U(3D))=\vol(1+3\ZZ_3\omega) \prod\limits_{p|D}\vol(\ZZ+D\ZZ_p[\omega])=V_{3D}$ and, changing notation, we get:
 \[
 Z_f(s, \chi_{D, f}\varphi_f, \Phi_f)=V_{3D}\sum_{\alpha_f \in U(3D)\setminus\AAA_{K, f}^{\times}/K^{\times}} I(s, \alpha_f, \Phi_f)\chi_{D, f}(\alpha) \varphi_f(\alpha).
 \]

Finally together with (\ref{L_to_Z0}) we get the result of the lemma.  \end{proof}

\subsection{Representative classes of $\Cl(\OO_{3D})$}
We will use the following lemma (see \cite{RV}) that is easy to show:

\begin{Lem}
Any primitive ideal of $\OO_K$ can be be written in the form $\A=[a, \frac{-b+\sqrt{-3}}{2}]_{\ZZ}$ as a $\ZZ$-module, where $b$ is an integer (determined only modulo $2a$) such that  $b^2\equiv -3 \mod 4a$ and $\Nm \A=a$. 

Conversely, given an integer satisfying the above congruence and $\A$ defined as above, we get that $\A$ is an ideal in $\OO_K$ of norm $a$.
\end{Lem}

We will use the notation $k_{\A}$ for the generator $k_{\A}\equiv 1\mod 3$ of a primitive ideal $\A$ in $\OO_K$. If we choose a lattice such that $\A=[a, \frac{-b+\sqrt{-3}}{2}]_{\ZZ}$, we denote the corresponding CM point $\tau_{\A}=\frac{-b+\sqrt{-3}}{2a}.$

\bigskip

We can write adelically $\Cl(\OO_{3D}) \simeq U(3D)\setminus\AAA_{K, f}^{\times}/K^{\times}$. This follows from the Strong approximation theorem, as $K$ is a PID and thus we have $\ds U(3D)\setminus\AAA_{K, f}^{\times}/K^{\times}\cong (\prod\limits_{p|3D}(\ZZ_p[\omega])^{\times}/(\ZZ+3D\ZZ_p[\omega])^{\times})/\left<-\omega\right>.$  Then we can define the map $\ds \prod_{p|3D}(\ZZ_p[\omega])^{\times}/(\ZZ+3D\ZZ_p[\omega])^{\times}/\left<-\omega\right> \ra I(3D)/P_{\ZZ, 3D}$ given by
\[
(\alpha_v)_{v|3D} \ra (k_{\alpha}),
\] 
where we choose the representative $\pm \alpha \omega^{k}$, such that $\alpha_3\equiv 1\mod 3$, and
 $k_{\alpha}$ is an element of $\OO_K$ such that $k_{\alpha}\equiv \alpha_v \mod 3D$. Note that this is well defined as $(k_{\alpha})$ gives us a unique class in $\Cl(\OO_{3D})$, and two elements $\alpha_1, \alpha_2$ get sent to the same class in $\Cl(\OO_{3D})$ only if $\alpha_1\equiv \alpha_2 \mod 3D$. 
 
 \bigskip
Thus for $\alpha_f \in \widehat{\OO}_K$  we can choose a class $[\A_{\alpha}]$ in  $\Cl(\OO_{3D})$ by taking a representative $\A_{\alpha_f}=(k_{\alpha})$, for $k_{\alpha}\in \OO_{K}$ such that $k_{\alpha} \equiv \alpha_p \mod 3D\ZZ_p[\omega]$ for $p|3D$.  Note that this choice is not unique. However, we can pick the representatives $\A_{\alpha}$ to be primitive ideals.




Thus we can further write $\A_{\alpha}$ as a $\ZZ$-lattice $\A_{\alpha}=[a, \frac{-b+\sqrt{-3}}{2}]_{\ZZ}$, where $a=\Nm\A_{\alpha}$ and $b$ is chosen (not uniquely) such that $b^2\equiv -3 \mod 4a$. We define the corresponding CM point $\tau_{\A_{\alpha}}=\frac{-b+\sqrt{-3}}{2a}$.

\subsection{Eisenstein series of weight $1$}

We will now connect $I(s, \alpha_f, \Phi_f)=\sum_{k\in K^{\times}} \frac{k}{|k|_{\CC}^{2s}} \Phi_f(k\alpha_f)$ to an Eisenstein series. We define the following classical Eisenstein series of weight $1$: 
\[
E_{\varepsilon}(s, z)=\sum\limits_{m, n}{'} \frac{\varepsilon(n)}{(3mz+n)|3mz+n|^s}.
\]
Here the sum is taken over all $m, n \in \ZZ$ except for the pair $(0, 0)$, and $\varepsilon=\left(\frac{\cdot}{3}\right)$ is the quadratic character associated to the field extension $K/\QQ$. The Eisenstein series $E_{\varepsilon}(s, z)$ does not converge absolutely for $s=0$, but we can still compute its value using the Hecke trick (see \cite{H}). We compute its Fourier expansion at $s=0$ in the following section.

Using this notation, we have the following equality:

\begin{Lem} For $\alpha_f \in \prod\limits_{v\nmid \infty} \OO_{K_v}^{\times}$, let $\A_{\alpha_f}=(k_{\alpha})$ be a choice of an ideal in the corresponding class of $\Cl(\OO_{3D})$. We write $\A_{\alpha}=[a_{\alpha}, \frac{-b+\sqrt{-3}}{2}]_{\ZZ}$ and take $\tau_{\A_{\alpha}}=\frac{-b+\sqrt{-3}}{2a_{\alpha}}$ the corresponding CM point. Then we have:
\[ I(s, \alpha_f, \Phi_f)= \frac{1}{2}\frac{(\Nm\A_{\alpha})^{1-s}}{k_{\alpha}} E_{\varepsilon}(s, D\tau_{\overline{\A_{\alpha}}}).
\]
\end{Lem}	

\begin{rmk} Note that the variable $\tau_{\overline{\A_{\alpha}}}$ on the left hand side  is not uniquely defined. However, the function is going to be invariant on the class $[\A_{\alpha}]$ in $\Cl(\OO_{3D})$.

\end{rmk}

\bigskip

\begin{proof} Recall that $\ds I(s, \alpha_f, \Phi_f)=\sum_{k\in K^{\times}} \frac{k}{|k|_{\CC}^{2s}} \Phi_f(k\alpha_f)$. We need to compute $\Phi_f(k\alpha_f)$. Note that $\Phi_v(k\alpha_v)\neq 0$ only for $k\alpha_v \in \OO_{K_v}$ at all places $v$, and since $\alpha_v \in \OO_{K_v}^{\times}$, we must have $k\in \OO_{K_v}$ as well for all $v$. This implies $k\in \OO_{K}$ and for all places $v\nmid 3D$ we get $\Phi_v(k\alpha_v)=1$ for $k\in \OO_K$. Thus we can rewrite:
\[
I(s, \alpha_f, \Phi_f)=\sum_{k\in \OO_K} \frac{k}{|k|_{\CC}^{2s}} \Phi_{3D}(k\alpha_{3D}),
\]
\noindent where $\Phi_{3D}=\prod_{v|3D} \Phi_v$ and $\alpha_{3D}=(\alpha_v)_{v|3D}$.

We can further compute $\Phi_v(k\alpha_v)$ for $v|3D$. Recall that we defined $\Phi_p=\Char_{(\ZZ+3D\ZZ_p[\omega])^{\times}}$ for $p|D$ and $\Phi_3=\Char_{(1+3\ZZ_3[\omega])^{\times}}$. Then we have $\Phi_{3D}(k\alpha_{3D})\neq 0$ iff $k\alpha_p \in a+3D\ZZ_p[\omega]$ for some integer $a$, $(a, p)=1$ and, for $p=3$, $k\alpha_{3}\in 1+3\OO_{K_3}$. 

Recall that we defined $k_{\alpha}$ such that $k_\alpha \equiv \alpha_p \mod 3D\ZZ_p[\omega]$ for all $p|3D$. Then $kk_{\alpha} \in a+3D\ZZ_p[\omega]$ for $(a, p)=1$ and $kk_{\alpha}\in 1+3\ZZ_3[\omega]$. Furthermore, for $k\in \OO_K$ we actually have $\Phi_{3D}(k\alpha_{3D})=\Phi_{3D}(kk_{\alpha})$. Then we can rewrite $I(s, \alpha_f, \Phi_f)$ using $k_{\alpha}$ as $\ds I(s, \alpha_f, \Phi_f)=\sum\limits_{k\in \OO_K} \frac{k}{|k|_{\CC}^{2s}} \Phi_{3D}(kk_{\alpha}).$ We can rewrite this further:

\[
I(s, \alpha_f, \Phi_f)=\frac{|k_{\alpha}|_{\CC}^{2s}}{k_{\alpha}}\sum_{k\in \OO_{K}} \frac{kk_{\alpha}}{|kk_{\alpha}|_{\CC}^{2s}} \Phi_{3D}(kk_{\alpha}),
\]

Finally, we will make this explicit. Note that we must have $kk_{\alpha} \in \A_{\alpha}$, where $\A_{\alpha}=(k_{\alpha})$, as well as $kk_{\alpha}\in a_p+D\ZZ_p[\omega]$ for some integer $a_p$, $(a_p, p)=1$, and $kk_{\alpha}\in 1+3\ZZ_3[\omega]$. By the Chinese remainder theorem, we can find an integer $a$ such that $a\equiv a_p \mod D$ and $a\equiv 1 \mod 3$. Then $kk_{\alpha}\in (a+D\prod\limits_{p|3D}\ZZ_p[\omega])  \cap \OO_K$, thus $kk_{\alpha} \in P_{\ZZ, 3D} \cap P_{1, 3}$. Here $P_{\ZZ, 3D}=\{k \in K: k\equiv a\mod 3D\OO_{K} ~\text{for some integer}~ a, (a, 3D)=1\}$ and  $  P_{1, 3}=\{k\in K: k\equiv 1 \mod 3\}$. We rewrite: 
\[
I(s, \alpha_f, \Phi_f)=\frac{|k_{\alpha}|_{\CC}^{2s}}{k_{\alpha}}\sum_{k\in \A_{\alpha} \cap P_{\ZZ, D}\cap P_{1, 3}} \frac{k}{|k|_{\CC}^{2s}}.
\]

Finally, we want to write the elements of $ \A_{\alpha} \cap P_{\ZZ, D}\cap P_{1, 3}$ explicitly. Recall that we can write $\A_{\alpha}$ as a $\ZZ$-lattice $\A_{\alpha}=[a, \frac{-b+\sqrt{-3}}{2}]_{\ZZ}$. Then all of the elements of $\A$ are of the form $ma+n\frac{-b+\sqrt{-3}}{2}$ for some integers $m, n\in \ZZ$. Moreover, note that the intersection of $\A$ and $P_{\ZZ, 3D}=\{k \in \OO_K: k \equiv n \mod {3D}, \text{for some integer}~ n, (n, 3D)=1\}$ is $\{ma+3Dn\frac{-b+\sqrt{-3}}{2}: m, n \in \ZZ\}$. Further taking the intersection with $P_{1, 3}$, we must have $ma\equiv 1$, thus, as $a$ is norm in $\OO_K$, $m \equiv 1 \mod 3$, and we can rewrite $I(s, \alpha_f, \Phi_f)$ in the form:
\[
I(s, \alpha_f, \Phi_f)
=
\frac{a^s}{k_{\alpha}}\sum\limits_{m, n \in \ZZ, m\equiv 1 (\text{mod} ~3)} 
\frac{ma+n\frac{-b+\sqrt{-3}}{2}}{|ma+3nD\frac{-b+\sqrt{-3}}{2}|_{\CC}^{2s}}.
\]

By changing $n\ra -n$ and taking out a factor of $a^{1-2s}$, we have:
\[
I(s, \alpha_f, \Phi_f)
=
\frac{a^{1-s}}{k_{\alpha}}\sum\limits_{\substack{m, n \in \ZZ, \\ m\equiv 1 (\text{mod} ~3)}} 
\frac{1}{(m+n\frac{b+\sqrt{-3}}{2a})|m+3nD\frac{b+\sqrt{-3}}{2a}|_{\CC}^{2s-2}}.
\]

Note that for $Re(s)>1$ the integral converges absolutely, and we can rewrite the sum as:
\[
I(s, \alpha_f, \Phi_f)
=
\frac{1}{2}\frac{a^{1-s}}{k_{\alpha}}\sum\limits_{m, n \in \ZZ} 
\frac{\varepsilon(m)}{(m+3nD\frac{b+\sqrt{-3}}{2a})|m+3nD\frac{b+\sqrt{-3}}{2a}|_{\CC}^{2s-2}},
\]

\noindent where $\varepsilon(m)=\left(\frac{m}{3}\right)$ is the usual quadratic character. On the right hand side we recognize the Eisenstein series $E_{\varepsilon}(2s-2, \tau_{\overline{\A_{\alpha}}})$ and  we get $I(s, \alpha_f, \Phi_f)
=
\frac{1}{2}\frac{a^{1-s}}{k_{\alpha}}E_{\varepsilon}(2s-2, D\tau_{\overline{\A_{\alpha}}}).$ By analytic continuation, we can extend the equality to all $s\in \CC$. \end{proof}

Now we can rewrite the linear combination in Lemma \ref{Z_linear} by taking representatives $\A$ for the classes of $\Cl(\OO_{3D})$. Note that for $\alpha \in \widehat{\OO}_K^{\times}$ with $\alpha\equiv 1\mod 3$ we have $\varphi_f(\alpha_f)=1$ and
\[
\chi_D(\alpha_f)=\chi_D((\alpha)_{p|3D})= \chi_D((k_{\alpha})_{p|3D})=\chi_D^{-1}((k_{\alpha})_{p\nmid 3D})=\overline{\chi_D(\A_{\alpha})}.
\]
Using the lemma above and after inverting each class $\A \ra\overline{\A}$ in $\Cl(\OO_{3D})$, we get:

\begin{cor}\label{Eisenstein_linear_comb} For all $s$, taking representative ideals $\A=[a, \frac{-b+\sqrt{-3}}{2}]_{\ZZ}$ for the classes in the ring class group $\Cl(\OO_{3D})$, we have:
\[
L_f(s, \chi_D\varphi)=\frac{1}{2}\sum_{[\A] \in \Cl(\OO_{3D})}E_{\varepsilon}(2s-2, D\tau_{\A})\chi_D(\A)\frac{(\Nm\A)^{1-s}}{k_{\overline{\A}}},
\]
where $\A=(k_{\A})$ with $k_{\A}\equiv 1\mod 3$ and $\tau_{\A}=\frac{-b+\sqrt{-3}}{2a}$ the associated CM points.
\end{cor}

\subsection{Fourier expansion of $E_{\varepsilon}(s, z)$ at $s=0$}

We want to connect the Eisenstein series  $\ds E_{\varepsilon}(s, z)=\sum_{c, d} {'}\frac{\varepsilon(d)}{(3cz+d)|3cz+d|^{2s}}$ to the theta function 
\[
\Theta_K(z)=\sum\limits_{m, n \in \ZZ}e^{2\pi i (m^2+n^2-mn)z}
\]
associated to the number field $K$. It is a modular form of weight $1$ for the congruence group $\Gamma_1(3)$. Note that this differs from the theta function $\Theta_K$ chosen by Rodriguez-Villegas and Zagier in \cite{RV-Z} by a factor of $1/2$. 

More precisely, we are going to show the following version of the Siegel-Weil theorem:

\begin{Thm}\label{siegel_weil} $\ds E_{\varepsilon}(0, z)=2L(1, \varepsilon)\Theta_K(z)$.
\end{Thm}

{\it Proof.} We will show this by computing the Fourier expansion of $E_{\varepsilon}(s, z)$ at $s=0$ using the Hecke trick and comparing it to the Fourier expansion of $\Theta_K(z)$. We will follow closely the exposition of Pacetti \cite{P}. This is also done by Hecke in \cite{H}. We first rewrite $E_{\varepsilon}(s, z)$ in the form:
\[
E_{\varepsilon}(z, s)
=
\sum_{d}{'}\frac{\varepsilon(d)}{d^{1+2s}}+2\sum\limits_{c=1}^{\infty}\sum\limits_{r=0}^2\frac{\varepsilon(r)}{3^{2s+1}}\sum_{d\in \ZZ}\frac{\varepsilon(r)}{(\frac{3cz+r}{3}+d)|\frac{3cz+r}{3}+d|^{2s}}.
\]
We define for $z$ in the upper-half plane $H(z, s)=\sum\limits_{m\in \ZZ}\frac{1}{(z+m)|z+m|^{2s}}$ and then we can rewrite the form above as:
\[
E_{\varepsilon}(s, z)=2L(\varepsilon, s)+2\sum\limits_{c=1}^{\infty} \sum\limits_{r=0}^{2} \frac{\varepsilon(r)}{3^{2s+1}}H\left(\frac{3dz+r}{3}, s\right).
\]

Pacetti (\cite{P}), following Shimura (Lemma 1, p. 84, \cite{Sh}), computed the Fourier expansion of $H(z, s)$ when $s\ra 0$ to be $\lim\limits_{s\ra 0} H(s, z)=-\pi i -2\pi i \sum\limits_{n=1}^{\infty} q^n$. It gives us in the relation above:
\[
E_{\varepsilon}(0, z)
=
2L(\varepsilon, 1)+2\sum\limits_{c=1}^{\infty} \sum\limits_{r=0}^{2} \frac{\varepsilon(r)}{3}(-\pi i -2\pi i\sum\limits_{n=1}^{\infty}  e^{2\pi i nzc} \omega^{nr}).
\]

We compute separately the inner sum and get:
\[
 \sum\limits_{r=0}^{2} \frac{\varepsilon(r)}{3}(-\pi i +\sum\limits_{n=1}^{\infty}  e^{2\pi i nzc} \omega^{nr})
=
-\frac{2\pi i}{3} G(\varepsilon)\sum\limits_{n=1}^{\infty}  e^{2\pi i nzc} \varepsilon(n),
\]
\noindent where $G(\varepsilon)=\sum\limits_{r=0}^2 \varepsilon(r)\omega^r=\sqrt{-3}$ is the quadratic Gauss sum corresponding to $\varepsilon$. Then we can rewrite:
\[
E_{\varepsilon}(0, z)
=
2L(\varepsilon, 1)+\frac{4\pi \sqrt{3}}{3}\sum\limits_{N=1}^{\infty}( \sum_{m|N} \varepsilon(m)  )e^{2\pi i Nz}.
\]
Since $\varepsilon$ is a quadratic character, we compute $L(1, \varepsilon)=\frac{\pi\sqrt{3}}{9}$ (see Kowalski \cite{Kow}) and this gives us the Fourier expansion:
\begin{equation}\label{expansion_eisenstein}
E_{\varepsilon}(0, z)
=
\frac{2\pi\sqrt{3}}{9} (1+6\sum\limits_{N=1}^{\infty} (\sum_{m|N} \varepsilon(m)  )e^{2\pi i Nz}).
\end{equation}

 It is actually easy to show that $\ds  \sum_{m|n}\varepsilon(m)$ represents the number of ideals of norm $n$ in $\OO_K$ and we can recognize the sum in the bracket on the RHS of (\ref{expansion_eisenstein}) to equal the theta function $\Theta_K(z)=1+6\sum\limits_{\A}e^{2 \pi i (\Nm \A) z}$, which finishes the proof.

\subsection{Formula for $L(1, \chi_D\varphi)$}

Applying Corollary \ref{Eisenstein_linear_comb} for $s=1$ we get $\ds L_f(1, \chi_D\varphi)
=
\frac{1}{2}\sum_{[\A] \in \Cl(\OO_{3D})}\frac{1}{\bar{k}_{\A}} E_{\varepsilon}(0, D\tau_{\A})  \chi_D(\A).$ Furthermore, from Theorem \ref{siegel_weil}, this is the same as:

\begin{equation}\label{linear_theta}
L_f(1, \chi_D\varphi)
=
\frac{\pi\sqrt{3}}{9}\sum_{[\A] \in \Cl(\OO_{3D})}
\frac{1}{\bar{k}_{\A}} \Theta_K(D\tau_{\A})  \chi_D(\A).
\end{equation}

We need one more step before rewriting the formula as a trace. We will use the following lemma:

\begin{Lem}\label{theta_omega} For $\A=[a, \frac{-b+\sqrt{-3}}{2}]_{\ZZ}$ a primitive ideal of norm $\Nm\A=a$, with generator $\A=(k_{\A})$, where $k_{\A}\equiv 1 \mod 3$ and $\tau_{\A}=\frac{-b+\sqrt{-3}}{2}$, we have: 
\[
\Theta_K\left(\tau_{\A}\right)=\overline{k_{\A}}\Theta_K\left(\omega\right).
\]
\end{Lem}

\begin{proof} Since $\A=[a, \frac{-b+\sqrt{-3}}{2}]_{\ZZ}$ as a $\ZZ$-lattice, we can write its generator $k_{\A} $ in the form $k_{\A}=ma+3n\frac{-b+\sqrt{-3}}{2}$ for some integers $m, n$ such that $m\equiv 1(3)$ and $\gcd(m, 3n)=1$. Then we can find integers $A, B$ such that $mA+3nB=1$, and thus $\left(\begin{smallmatrix}A & B \\ 3n & m \end{smallmatrix}\right)$ is a matrix in $\Gamma_1(3)$. Since $\Theta(z)$ is a modular form of weight $1$ for $\Gamma_1(3)$, we have $ \Theta_K\left(\left(\begin{smallmatrix}A & B \\ 3n & m \end{smallmatrix}\right)\tau_{\A}\right)=\left(m+3n\tau_{\A}\right)\Theta_K\left(\tau_{\A}\right)$.
 
Noting that $3n\tau_{\A}+m=k_{\A}/a=1/\overline{k_{\A}}$, we can compute the term on the LHS to be $\Theta_K((A\tau_{\A}+B)\overline{k_{\A}})$ and, after expanding, we are evaluating $\Theta_K$ at  $-3nA\frac{b^2+3}{4a}+abB+\frac{b(-mA+3nB)}{2}+\frac{\sqrt{-3}}{2}$. Note that $mA-3nB=1$ implies that $mA$ and $3nB$ have different parities. Also, $b$ is odd and $b^2+3\equiv 0 \mod 4a$. Then $-3nA\frac{b^2+3}{4a}+abB+\frac{b(-mA+3nB)+1}{2}\in \ZZ$ and thus using the period $1$ of $\Theta_K$ we get $\Theta_K\left(\left(\begin{smallmatrix}A & B \\ 3n & m \end{smallmatrix}\right)\tau_{\A}\right)=\Theta_K\left(\omega\right)$. This finishes the proof. \end{proof}

Using the Lemma above we can rewrite (\ref{linear_theta}) as:
\begin{equation}\label{final_form}
L_f(1, \varphi\chi_D) = \frac{\pi\sqrt{3}}{9}\Theta_K\left(\omega\right)\sum_{[\A] \in \Cl(\OO_{3D})}\frac{\Theta_K(D \tau_{\A})}{\Theta_K(\tau_{\A})}\chi_D(\A).
\end{equation}


Now we will rewrite the formula (\ref{final_form}) as a trace. We can define $\ds f(z)=\frac{\Theta_K(Dz)}{\Theta_K(z)}$ and this is a modular function for $\Gamma_0(3D)$. We will prove in Section \ref{Shimura} in Proposition \ref{Galois_conj_theta} that $f(\omega)\in H_{3D}$, the ring class field of corresponding to the order $\OO_{3D}$. Moreover, we show in the same proposition that, for $\A=[a, \frac{-b+\sqrt{-3}}{2}]_{\ZZ}$ a primitive ideal in $\OO_K$, we have the Galois conjugate: 
\[
f(\omega)^{\sigma_{\A}^{-1}}=f(\tau_{\A}),
\]
where $\sigma_{\A}$ is the Galois action corresponding to the ideal $\A$ via the Artin map.

Furthermore, from Corollary \ref{GaloisD} we have $(D^{1/3})^{\sigma_{\A}^{-1}}=D^{1/3}\chi_D(\A)$ and then formula (\ref{final_form}) becomes:

\begin{equation}
L(E_D, 1)
=
\frac{\pi\sqrt{3}}{9}D^{-1/3}\Theta_K(\omega)\sum_{[\A] \in \Cl(\OO_{3D})}\left(D^{1/3}\frac{\Theta_K(D\omega)}{\Theta_K(\omega)}\right)^{\sigma_{\A^{-1}}} 
\end{equation}

Moreover, $D^{1/3} \in H_{3D}$ (see for example Cohn \cite{Co}). Thus we can rewrite the sum on the left hand side as $\Tr_{H_{3D}/K} \left(D^{1/3}\frac{\Theta_K(D\omega)}{\Theta_K(\omega)}\right)$. We can compute the extra terms as well. Rodriguez-Villegas and Zagier in \cite{RV-Z} cite $\Theta_K\left(\frac{-9+\sqrt{-3}}{18}\right)=-6\Gamma\left(\frac{1}{3}\right)^3/(2\pi)^2$. Using several of the properties of $\Theta_K$ proved in the Appendix, we can compute $\Theta\left(\omega\right)=\Gamma\left(\frac{1}{3}\right)^3/(2\pi^2)$.

As the real period $\Omega_D$ of the elliptic curve $E_D$ is $\Omega_D=D^{-1/3}\frac{\sqrt{3}\Gamma\left(\frac{1}{3}\right)^3}{6\pi}$, we get the formula of Theorem \ref{thm_1_theta}:
\begin{equation}
L(E_D, 1)
=
\Omega_D\frac{1}{3}\Tr_{H_{3D}/K} \left(D^{1/3}\frac{\Theta_K(D\omega)}{\Theta_K(\omega)}\right).
\end{equation}

Note that this implies $S_D \in K$. Moreover, as $D^{1/3}\Theta(D\omega)/\Theta(\omega)$ is invariant under complex conjugation, we get $S_D\in \RR$ which furthermore implies $S_D\in \QQ$. We will show in Section \ref{int} that actually $3c_{3D}S_D \in \ZZ$.

\begin{rmk} If we take $D=D_1D_2^2$ such that $D_0=D_1D_2$ is square-free, note that the character $\chi_D=\chi_{D_1}\overline{\chi_{D_2}}$ is well defined on the class group $\Cl(\OO_{3D_0})$. Then the above computations work for $D_0$ and the character $\chi_{D}=\chi_{D_1}\overline{\chi_{D_2}}$ instead of  $\chi_{D_0}=\chi_{D_1}\chi_{D_2}$ and we get:
\begin{equation}
L(1, \chi_D\varphi)=\frac{\pi\sqrt{3}}{9}\Theta_K\left(\omega\right)\sum_{[\A] \in \Cl(\OO_{3D_0})}\frac{\Theta_K(D_0\tau_{\A})}{\Theta_K(\tau_{\A})}\chi_{D}(\A).
\end{equation}
Note that, for $D=D_1D_2^2$, $L(E_D, s)=L(s, \chi_{D_1}\overline{\chi_{D_2}}\varphi)$ and thus we have:
\begin{equation}\label{S_D}
S_D=\frac{1}{3}\sum_{[\A] \in \Cl(\OO_{3D_0})}\frac{\Theta_K(D_0\tau_{\A})}{\Theta_K(\tau_{\A})}\chi_{D}(\A)D^{1/3}.
\end{equation}

As before from Corollary \ref{GaloisD}  we have $(D_1^{1/3}D_2^{2/3})^{\sigma_{\A}^{-1}}=D_1^{1/3}\chi_{D_1}(\A)D_2^{2/3}\overline{\chi_{D_2}(\A)}$ and finally we can write the expression above as:
\[
L(E_D, 1)=\frac{\Omega_D}{3}\Tr_{H_{3D_0}/K}\left(D^{1/3}\frac{\Theta_K(D_0\omega)}{\Theta_K(\omega)}\right).
\]

\end{rmk}


\section{Second formula for $S_D$}

For $r\in \ZZ$, $\mu \in \{1/2, 1/6\}$, we define the theta functions of weight $1/2$:
\[
\theta_{r, \mu}(z)=\sum\limits_{n\in \ZZ}e^{\pi i (n+r/D-\mu)^2z} (-1)^n.
\]

Throughout the paper we will use the notation $r\in \ZZ/D\ZZ$ to mean any family of representatives for the residues $r \mod D$. We denote $\theta_0=\theta_{0, 1/6}$. Note that $\theta_0(z)=\eta(z/3)$, where $\eta$ is the Dedekind eta function, while $\sum\limits_{\substack{r\in \ZZ/D\ZZ\\ r\equiv 1 (6)}} \theta_{r, 1/6}(z)=\eta\left(\frac{z}{3D^2}\right)$.


In this section we will use a Factorization formula of Rodriguez-Villegas and Zagier from \cite{RV-Z2} to show the following theorem:

\begin{Thm}\label{thm_r} In the case of $D=\prod\limits_{p_i\equiv 1 (3)} p_i^{e_i}$, let $D_0=\prod\limits_{p_i|D} p_i$ be the radical of $D$ and $\sigma(D)$ the number of distinct prime divisors of $D$. Then $S_D$ is an integer square up to an even power of $3$ and we have:
\begin{equation}\label{eqn_r}
 S_D=\frac{(-1)^{\sigma(D)}}{3^{\sigma(D)+2}} T_D^2,
\end{equation}
where $T_D/3 \in \ZZ$ if $\sigma(D)$ is even and $T_D/\sqrt{-3}\in \ZZ$ if $\sigma(D)$ is odd. We have the exact formula:
\[
T_{D}=\frac{\omega^{k_0}}{\varphi(D_0)}\Tr_{H_{\OO}/K} \left(\frac{\theta_{1, 1/6}(\tau)}{\theta_0(\tau)}\overline{\pi_1}^{-2/3}\pi_2^{1/3}\right).
\]

Here $\tau=\frac{-b+\sqrt{-3}}{2}$ is a CM-point, with $b^2\equiv -3 \mod 12D^2$, $\pi_1, \pi_2$ are elements in $\OO_K$ such that $\pi_1, \pi_2\equiv 1\mod 3$, $\pi_1\pi_2$ has norm $D_0$ and $\pi_1\pi_2^2$ has norm $D$, and such that $(\pi_1\pi_2)^2$ divides the ideal $\left(\frac{-b+\sqrt{-3}}{2}\right)$, $H_{\OO}$ is the ray class field of modulus $3D_0$ and $\omega^{k_0}$ is the unique cube root of unity that makes $T_D$ real or purely imaginary.


\end{Thm}

Below we discuss the details of $D$ square-free. All definitions and proofs can be easily extended to all $D$. We do that in Section \ref{not_square_free}.

Take $\tau=\frac{-b+\sqrt{-3}}{2}$ a CM point such that $b^2\equiv -3 \mod 12 D^2$ and an element $\pi\equiv 1\mod 3$ of norm $D$ in $\OO_K$ such that $\pi^2$ divides the ideal $(\frac{-b+\sqrt{-3}}{2})$.

\bigskip
We will use the notation:

\begin{itemize}
\item  $\ds \Theta_{\mu}(\tau)= \begin{cases}
\frac{3}{2}\Theta_K\left(\tau\right)-\frac{1}{2}\Theta_K\left(\tau/3\right), &\text{~for~} \mu=1/6 \\
\Theta_K\left(\tau/3\right),& \text{~for~} \mu=1/2.
\end{cases}$

\item $\ds S_{\mu}=\frac{2}{3}\sum_{[\A]\in \Cl(\OO_{3D})}\frac{\Theta_{\mu}(D\tau_{\A})}{\Theta_K(\tau_{\A})} \chi_{D}(\A)D^{1/3}$.

\end{itemize}

In Theorem \ref{thm1} we have proved that $3c_{3D}S_D=S_{1/6}+1/2S_{1/2}.$ We are actually going to show in Corollary \ref{S_zero} that $S_{1/2}=0$, thus it is enough to compute the formula (\ref{eqn_r}) for $S_{1/6}$.

Using a Factorization formula of Rodriguez-Villegas and Zagier from \cite{RV-Z2} we will write the theta functions $\Theta_{\mu}$ of weight $1$ as linear combinations of products of theta functions of weight $1/2$ in Proposition \ref{fact_r}. We define:
\[
\ds R_{D, \mu}(z)=\sum_{\substack{r \in (\ZZ/D\ZZ)^{\times} \\ r\equiv 1 (6)}} \frac{\theta_{r, \mu}(3z)}{\theta_0(3z)}\chi_{\pi}(r).
\]

We show in Lemma \ref{z_over_3} that $S_{1/6}=|R_{D, 1/6}(\tau)D^{-1/3}|^2$. Moreover, if we denote
\[
T_D=R_{D, 1/6}(\tau/3) \overline{\pi}^{-2/3} \omega^{k_0},
\]
for a cubic root of unity $\omega^{k_0}$, then $S_{1/6}=|T_{D, 1/6}(\tau)|^2$.

We show in Lemma \ref{trace_H} that $\ds T_D=\frac{1}{\varphi(D)}\Tr_{H_{\OO}/K} \frac{\theta_{1, 1/6}(\tau)}{\theta_0(\tau)} \overline{\pi}^{-2/3}\omega^{k_0}$ and that $T_{D} \in K$. Furthermore, we show in proposition \ref{M_D} that $T_D=(-1)^{\sigma(D)}\overline{T_D}$ and thus $T_D\in \QQ$ or $T_D/\sqrt{-3} \in \QQ$ and thus
\[
S_D=\frac{(-1)^{\sigma(D)}}{3c_{3D}}T_D^2.
\]

Moreover, in Section \ref{int} we show that $3c_{3D} S_D$ is an integer, hence $T_D/3\in \ZZ$ for $\sigma(D)$ even and $T_D/\sqrt{-3}\in \ZZ$ for $\sigma(D)$ odd. 

Finally, for $D$ a product of split primes, we have $S_D\neq 0$ only for $D\equiv 1\mod 9$. In this case the Tamagawa numbers equal $c_{3D}=3^{1+\sigma(D)}$, thus we have:
\[
S_D=
\begin{cases}
(T_D/3^{\sigma(D)/2+1})^2, & \text{ for } \sigma(D) \text{ even}, \\
((T_D/\sqrt{-3})/3^{(\sigma(D)+1)/2} )^2, & \text{ for } \sigma(D) \text{ odd}.
\end{cases}
\]
Hence $S_D$ is an integer square up to an even power of $3$ and this finishes the proof of Theorem \ref{thm_r}.




\subsection{Factorization lemma}

As in the previous section, we write a primitive ideal $\A$ as a lattice $\A=[a, \frac{-b+\sqrt{-3}}{2}]_{\ZZ}$ for $a=\Nm(\A)$ and $b^2\equiv -3\mod 4a$. We also define the CM point $\tau_{\A}=\frac{-b+\sqrt{-3}}{2a}$ corresponding to the $\ZZ$-lattice. We also denote by $k_{\A}$ the generator of $\A$ such that $k_{\A}\equiv 1 (3)$ and we write the generator in the form $k_{\A}=n_a a+ m_a \tau_{\A}$.

We start by recalling the Factorization Formula of Rodriguez-Villegas and Zagier (\cite{RV-Z2}, Theorem, page $7$) in the simplified case of $\alpha=p=0$:

\begin{Thm}\textbf{(Factorization formula.)}\label{factorization_formula_RVZ} For $a\in \ZZ_{>0}$, $\mu, \nu\in\QQ $, $z=x+yi\in \CC$, we have:
\begin{equation}\label{factorization_formula}
\sum_{m,n \in \ZZ} 
e^{2\pi i (m \nu+ n\mu)} 
e^{\pi(imn -\frac{|mz-n|^2}{2y})/a}
=
\sqrt{2ay}
\theta\left[\begin{matrix}a\mu \\ \nu\end{matrix}\right](a^{-1}z)\cdot 
\theta\left[\begin{matrix} \mu \\ -a\nu \end{matrix}\right](-a\bar{z}),
\end{equation}
where $ \theta\left[\begin{matrix}\mu \\ \nu \end{matrix}\right](z)=\sum\limits_{n\in \ZZ+\mu}e^{\pi in^2z+2\pi i \nu n}$ is a theta function of weight $1/2$.

\end{Thm}

A direct application of this is the following:

\begin{Lem}\label{sum_rD} With notation as above, we have:
\[
\sum_{r\in \ZZ/D\ZZ}
\frac{\sqrt{2ay}}{\sqrt{D}} 
\theta\left[\begin{matrix} a\mu+\frac{ar}{D} \\ \nu \end{matrix}\right]\left(D\frac{z}{a}\right)
\theta\left[\begin{matrix} \mu+\frac{r}{D}\\ -a\nu \end{matrix}\right]\left(-aD\overline{z}\right)
=
\sum_{m,n\in \ZZ} e^{2\pi i (m\nu+nD\mu)}e^{\pi  (mni-\frac{|n-mz|^2}{2y})\frac{D}{a}}
\]
\end{Lem}

\begin{proof} Plugging in $\mu:=\mu+\frac{r}{D}$, $z:=Dz$ in (\ref{factorization_formula}) and summing up for $r$ in $\ZZ/D\ZZ$, we get:
\[
\sum_{r\in \ZZ/D\ZZ}\sqrt{2aDy}\theta\left[\begin{matrix} a\mu+\frac{ar}{D} \\ \nu \end{matrix}\right]\left(\frac{Dz}{a}\right)\theta\left[\begin{matrix} \mu+\frac{r}{D} \\ -a\nu \end{matrix}\right]\left(-a\overline{Dz}\right)
=
\sum_{r\in \ZZ/D\ZZ}\sum_{m,n\in \ZZ} e^{2\pi i (m\nu+n\mu+nr/D)}e^{\pi  (mni-\frac{|n-mDz|^2}{2Dy})\frac{1}{a}}
\]

Exchanging the two sums on the RHS we get $\ds \sum_{m,n\in \ZZ}e^{2\pi i (m\nu+n\mu)}e^{\pi  (mni-\frac{|n-mDz|^2}{2Dy})\frac{1}{a}} \sum_{r\in \ZZ/D\ZZ} e^{2\pi i nr/D}.$ The inner sum $\sum\limits_{r\in \ZZ/D\ZZ} e^{2\pi i nr/D}$ equals $D$ when $D|n$, and $0$ otherwise, thus we are only summing over the integers $n$ that are multiples of $D$. Rewriting $n=Dn'$,  after simplifying we get the result of the lemma.\end{proof}

\bigskip

Using the lemma above and the notation $\ds\theta_{r, \mu}(z)=\sum\limits_{n\in\ZZ}(-1)^n e^{\pi i \left(n+\frac{r}{D}-\mu\right)^2 z}$, we show:

\begin{Prop}\label{fact_r} For ideals $\A=[a, \frac{-b+\sqrt{-3}}{2}]_{\ZZ}, \A_1=[a_1, \frac{-b+\sqrt{-3}}{2}]_{\ZZ}$ and $b$ such that $b^2\equiv -3 \mod 4D^2a^2a_1$, we have:
\[
\Theta_{\mu}\left(D\tau_{\A}\right)
=
\frac{\sqrt[4]{3}e^{\pi i (a-1)/6 } }{D\sqrt{a_1}}\sum_{r\in \ZZ/D\ZZ}\theta_{ar, \mu}\left(\tau_{\A^2\A_1}\right)\overline{\theta_{r, \mu}\left(\tau_{\A_1}\right)}.
\]

\end{Prop}

{\it Proof.} We apply Lemma \ref{sum_rD} for $\mu=-1/6$ and $\nu=1/2$, $D$ odd, $z=\frac{-b+\sqrt{-3}}{2Daa_1}$. It is easy to see on the LHS of the equation that we have $\theta\left[\begin{smallmatrix}-1/6+\frac{r}{D}\\ -a/2 \end{smallmatrix}\right]\left(z\right)
=
e^{-a\pi ir/D} e^{a\pi i/6}\theta_{r, \mu}(z),
$ and, as $a\equiv 1  (\text{mod}~ 6)$, also $\theta\left[\begin{smallmatrix} -a/6+\frac{ar}{D} \\ 1/2 \end{smallmatrix}\right]\left(z\right)=e^{\pi i ar/D} e^{-\pi i/6} \theta_{ar, \mu}(z)$. Moreover, since $D\equiv 1  (\text{mod}~ 6)$ we simplify the term $e^{-2\pi i nD/6}=e^{-2\pi i n/6}$ and  we can also compute $\ds \frac{\sqrt{2ay}}{\sqrt{D}}=\frac{\sqrt[4]{3}}{D\sqrt{a_1}}$. We get:
\begin{equation}\label{part_left}
\frac{\sqrt[4]{3}}{D\sqrt{a_1}e^{\pi i (a-1)/6} }\sum_{r\in \ZZ/D\ZZ} \theta_{ar, \mu}\left(\tau_{\A^2\A_1}\right)\overline{\theta_{r, \mu}\left(\tau_{\A_1}\right)}
=
\sum_{m,n\in \ZZ} e^{2\pi i (m/2-n/6)}e^{\pi  (mni-\frac{|nDaa_1-m\frac{-b+\sqrt{-3}}{2}|^2}{Daa_1\sqrt{3}})\frac{D}{a}}.
\end{equation}

Now we only have to show that the RHS equals $\Theta_{\mu}(D\tau_{\A})$. We claim:
\[
\ds e^{2\pi i (m/2+n/2)}e^{\pi  (mni-\frac{|naa_1D-m\frac{-b+\sqrt{-3}}{2}|^2}{aa_1D\sqrt{3}})\frac{D}{a}}=e^{2\pi i  \frac{|naa_1D-m\frac{-b+\sqrt{-3}}{2}|^2}{aa_1D}D\frac{-b+\sqrt{-3}}{6a}}.
\] 

Since the absolute values of the two sides already agree, we only need to show that the arguments agree as well, meaning $ 2\pi i \left(\frac{m}{2}+\frac{n}{2} +\frac{Dmn}{2a}\right) \equiv -2\pi i  \frac{|naa_1D-m\frac{-b+\sqrt{-3}}{2}|^2}{aa_1D}D\frac{b}{6a} \mod 2\pi i \ZZ.$ This is equivalent to showing that $\left(\frac{m^2}{2}+\frac{n^2}{2}\right) - \left(\frac{m^2}{2}Db\frac{(b^2+3)}{12a^2a_1}-Dmn\frac{(b^2+3)}{6a}+D\frac{n^2}{2}\frac{b}{3}\right)\in  \ZZ$,  and this follows easily from the conditions on $a, a_1, b$ and $D$.

Finally, we claim that $\ds \Theta_{\mu}(z)=\sum_{m,n \in \ZZ} e^{2\pi in(\mu+1/2)}e^{2\pi i\frac{|m\cdot \frac{b+\sqrt{-3}}{2}+naa_1|^2}{aa_1}z}$, which would finish the proof. This is immediate for $\mu=1/2$. For $\mu=1/6$ denote $\ds E_{*, k}(z)=\sum\limits_{\substack{m,n \in \ZZ\\ n\equiv k (3)}} e^{2\pi in/3}e^{2\pi i\frac{|m\cdot \frac{b+\sqrt{-3}}{2}+naa_1|^2}{aa_1} z}.$ Then we can write 
\[
\sum_{m,n \in \ZZ} e^{2\pi in/3}e^{2\pi i\frac{|m\cdot \frac{b+\sqrt{-3}}{2}+naa_1|^2}{aa_1}z}
=
E_{*, 0}(z)+\omega E_{*, 1}(z)+\omega^2 E_{*, 2}(z).
\] 
Note that $E_{*, 0}(z)=\Theta_K(3z)$ and $E_{*, 1}=E_{*, 2}$, as we can change $n\ra -n, m \ra -m$ in the Fourier expansion. Thus we get on the RHS the term $\Theta(3z)+(\omega+\omega^2)E_{*, 1}(z)=\Theta(3z)-E_{*, 1}(z)$. Furthermore $\Theta(z)=E_{*, 1}(z)+E_{*, 2}(z)+E_{*, 0}(z)$, thus we get $E_{*, 1}(z)=\frac{1}{2}(\Theta(z)-\Theta(3z)).$ Plugging in $E_{*, 1}(z)$ above we get the result of the proposition.

\bigskip
 A particular case of Lemma \ref{fact_r} is for $D=1$. As $\Theta\left(\tau_{A}/3\right)=0$ from Lemma \ref{theta_03} from the Appendix, we get:

\begin{cor}\label{theta_0} For $b^2\equiv -3\mod 12a^2a_1$ and $\A, \A_1$ as above, we have 
\[\Theta_{K}\left(\tau_{\A}\right)
=
\frac{2}{3}\frac{\sqrt[4]{3}}{\sqrt{a_1}}e^{\pi i (a-1)\frac{1}{6} } \theta_{0}\left(\tau_{\A^2\A_1}\right)\overline{\theta_0\left(\tau_{\A_1}\right)}.
\]
\end{cor}

\bigskip

Let $\ds f_{r, \mu}(z)=\frac{\theta_{r, \mu}(z)}{\theta_0(z)}$. Taking the ratios of the theta functions in Proposition \ref{fact_r} and Corollary \ref{theta_0} we get:

\begin{cor}\label{ratio} Under the same conditions as above, we have:
\begin{equation}
\frac{\Theta_{\mu}(D\tau_{\A})}{\Theta\left(\tau_{\A}\right)}
=
\frac{3/2}{D}\sum\limits_{r\in \ZZ/D\ZZ}
f_{ar, \mu}\left(\tau_{\A^2\A_1}\right) \overline{f_{r, \mu}\left(\tau_{\A_1}\right)}.
\end{equation}
\end{cor}

\bigskip

We are interested in the Galois conjugates of $f_{r, \mu}(\tau)$ for $\tau=\frac{-b+\sqrt{-3}}{2}$ such that $b^2\equiv -3 \mod 12D^2$. For $\A=[a, \frac{-b+\sqrt{-3}}{2}]_{\ZZ}$ a primitive ideal and $k_{\A}\equiv 1\mod 3$ its generator, we write $k_{\A}$ in the form $k_{\A}=n_{a}a+m_{a}\frac{-b+\sqrt{-3}}{2}$ with $3|m_a$ and $n_a\equiv 1(3)$. In Section \ref{shimura_r} we will show in Proposition \ref{Galois_theta_r} that
\[
f_{r, \mu}(\tau)^{\sigma_{\A}^{-1}}=f_{n'_a r, \mu}(\tau),
\]
where $\sigma_{\A}$ is the Galois action corresponding to the ideal $\A$ via the Artin map and $n'_a\equiv n_a (3D)$ with $n'_a$ odd. We also show
in Lemma \ref{galois_conj} in the same section that we have we have $
(f_{r, \mu}(\tau))^{\sigma_{\A}^{-1}}=f_{r, \mu}(\tau_{\A})$, thus we get the following lemma:

\begin{Lem}\label{theta_galois} For an ideal $\A=[a, \frac{-b+\sqrt{-3}}{2}]$ generated by $n_a a+m_a\frac{-b+\sqrt{-3}}{2}$ such that $m_a \equiv 0\mod 3$, $n_a\equiv 1 \mod 3$ and $b^2\equiv -3 \mod 12aD^2$, we have: 
\[
f_{r, \mu}(\tau)^{\sigma_{\A}^{-1}}
=
f_{r, \mu}\left(\tau_{\A}\right)
=
f_{n'_ar, \mu}(\tau)
\]
for $n'_a\equiv n_a (3D)$ with $n'_a$ odd.
\end{Lem}

Using the lemma above, we can rewrite Corollary \ref{ratio}:

\begin{cor}\label{ratio_n_a} Under the same conditions as above, for $\A=(n_aa+m_a\frac{-b+\sqrt{-3}}{2})$, $\A_1=(n_{a_1}a_1+m_{a_1}\frac{-b+\sqrt{-3}}{2})$ with $b^2\equiv -3 \mod 12a^2a_1D^2$,  we have:
\[
\frac{\Theta_{\mu}\left(D\tau_{\A}\right)}{\Theta\left(\tau_{\A}\right)}
=
\frac{3/2}{D}\sum\limits_{r\in \ZZ/D\ZZ}
f_{n_a'^2n'_{a_1}ar, \mu}\overline{
f_{n_{a'_1}r, \mu}\left(\tau\right)},
\]
where $n'_a\equiv n_a (3D)$, $n'_{a_1}\equiv n_{a_1} (3D)$ and $n'_a, n'_{a_1}$ odd.

\end{cor}

\subsection{$S_D$ as an absolute value}

In the following we will use Corollary \ref{ratio_n_a} for a choice of representative ideals for the classes of the ring class group $\Cl(\OO_{3D})$. We show first:

\begin{Prop}\label{absolute_value} For $\tau=\frac{-b+\sqrt{-3}}{2}$ such that $b^2\equiv -3\mod 12D^2$ and $\pi\equiv 1 \mod 3$ an element of norm of $D$ such that $(\pi)^2$ divides $(\frac{-b+\sqrt{-3}}{2})$, we have:
\[
S_{\mu}
=
D^{-2/3}|\sum\limits_{\substack{s\in (\ZZ/D\ZZ)^{\times} \\ s\equiv 1 (6)}}f_{s, \mu}\left(\tau\right) 
\chi_{\pi}(s) |^2.
\]
\end{Prop}

{\it Proof.} The structure of the ring class group of conductor $3D$ for $D=\prod\limits_{p_i\equiv 1\mod 3} p_i$ is given by  $\Cl(\OO_{3D})\cong (\ZZ/D\ZZ)^{\times}$ (see for example Cox \cite{Cox}). 
We will choose as representatives for the classes of $\Cl(\OO_{3D})$ ideals $\A_s$ such that $\Nm\A_s \equiv s\mod D$. For $b$ fixed, $b^2\equiv -3\mod 12D$, we take:
\[
\A_s=(n_s a_s + m_s\frac{-b+\sqrt{-3}}{2}), 
\]
where $a_s=\Nm(\A_s)\equiv s\mod D$, $n_s\equiv 1\mod 3D$, $m_s\equiv 0 \mod 3$. Note that this gives us $m\equiv b^{-1}(s-1) \mod 3D$. Moreover, it is easy to check that the ideals $\A_s$ for $s\in(\ZZ/D\ZZ)^{\times}$ are in different classes in $\Cl(\OO_{3D})$.

We take as before $\pi$ the element of norm $D$ such that $(\pi)^2$ divides the ideal $(\tau)=(\frac{-b+\sqrt{-3}}{2})$.  Then note that $\chi_{D}(\alpha)=\chi_{\pi}(\frac{\alpha}{\overline{\alpha}})=\chi_{\pi}(\frac{\alpha^2}{|\alpha|^2})$. As $b\equiv \sqrt{-3} \mod \pi$, we get $\alpha_s=n_s a_s + m_s\frac{-b+\sqrt{-3}}{2} \equiv s \mod \pi$ and thus $\chi_{\pi}((\alpha_s))=\chi_{\pi}(s^2/s)=\chi_{\pi}(s)$.

Taking representatives $s\in \ZZ/D\ZZ, s\equiv 1\mod 6$, we get $m_s\equiv 0\mod 6$ and $n_s\equiv 1\mod 6$. Summing up over $r\in \ZZ/D\ZZ$ with $r\equiv 1\mod 6$ and taking $\A_1=(1)$ in Corollary \ref{ratio_n_a}, we get:

\begin{equation}\label{sum_rs_0} 
\frac{\Theta_{\mu}\left(D\tau_{\A_s}\right)}
{\Theta\left(\tau_{\A_s}\right)}
=
\frac{3/2}{D}\sum\limits_{\substack{r\in \ZZ/D\ZZ \\ r\equiv 1 (6)}}f_{sr, \mu}\left(\tau\right)\overline{f_{r, \mu}\left(\tau\right)},
\end{equation}

Summing up for all $\{s\in (\ZZ/D\ZZ)^{\times}, s\equiv 1 (6)\}$ and rearranging the terms, we get:

\[
S_{\mu}
=
D^{-2/3}\sum_{\substack{s\in (\ZZ/D\ZZ)^{\times} \\ s\equiv 1 (6)}}\sum\limits_{\substack{r\in \ZZ/D\ZZ, \\ r\equiv 1(6)}} f_{sr, \mu}\left(\tau\right) \chi_{\pi}(rs)\cdot
\overline{f_{r, \mu}\left(\tau\right)\chi_{\pi}(r)}.
\]

Finally, we will further modify the sums on the RHS in order to sum up over $r\in (\ZZ/D\ZZ)^{\times}$ as well. In order to emphasize the dependence of $\theta_{r, \mu}$ on $D$ we will use the notation $f_{r/D}(z)=\frac{\theta_{r, \mu}(z)}{\theta_{0}(z)}= \frac{
\sum\limits_{n\in \ZZ}e^{\pi i \left(n+\frac{r}{D}-\mu\right)^2z}(-1)^n}{\theta_0(z)}$. Moreover, for $p_{i_1}\dots p_{i_k}|D$, denote:

\[
\ds S_{p_{i_1}\dots p_{i_k}}
=
\sum_{\substack{s\in (\ZZ/D\ZZ)^{\times} \\ s\equiv 1(6)}} \chi_{D}(\A_s)D^{-2/3}
\sum\limits_{\substack{r\in \ZZ/D\ZZ\\ r\equiv1 (6) \\ p_{i_1}\dots p_{i_k}|r}}
f_{sr/D} (\tau)
\overline{f_{r/D}(\tau)}.
\]

We claim that for $k\geq 1$ we have $S_{p_{i_1}\dots p_{i_k}}=0$. Note that we can rewrite 
\[
S_{\mu}=\sum_{p_i|D} S_{p_i}
-\sum_{p_ip_j|D}S_{p_ip_j}+\dots+(-1)^{n-1}S_{p_1, \dots, p_n}+\sum_{\substack{s\in (\ZZ/D\ZZ)^{\times} \\ s\equiv 1 (\text{mod}~ 6)}}\sum\limits_{\substack{r\in (\ZZ/D\ZZ)^{\times}, \\ r\equiv1  (\text{mod}~ 6)}} f_{sr/D}(\tau)\overline{f_{r/D}(\tau)}\chi_{D}(\A_s)D^{-2/3},
\]
thus showing $S_{p_{i_1}\dots p_{i_k}}=0$ for $k\geq 1$ proves our result.

To see that $S_{p_{i_1}\dots p_{i_k}}=0$, let $D'=D/(p_{i_1}\dots p_{i_k})$.  We recognize each of the inner sums $\ds \sum\limits_{\substack{r'\in \ZZ/D'\ZZ\\ r\equiv1 (\text{mod}~ 6)}}
f_{sr'/D'} (\tau)
\overline{f_{r'/D'}(\tau)}$
  of $S_{p_{i_1}\dots p_{i_k}}$ to be equal to $\ds \frac{D'}{3/2}\frac{\Theta_{\mu}\left(D'\tau_{\A_s}\right)}{\Theta\left(\tau_{\A_s}\right)}$ from (\ref{sum_rs_0}) for $D:=D'$.
  
 Denote $m=D/D'$. From the properties of the cubic character, we have $\chi_D=\chi_m \chi_{D'}$. Moreover, from our choice of ideals, we have $\ds\frac{\Theta_{\mu}\left(D'\tau_{\A_s}\right)}{\Theta\left(\tau_{\A_s}\right)}=\frac{\Theta_{\mu}\left(D'\tau_{\A_{s'}}\right)}{\Theta\left(\tau_{\A_{s'}}\right)}$ for $s\equiv s' \mod 3D'$, as $\A_s$ and $\A_{s'}$ are in the same class in $\Cl(\OO_{3D'})$. Then we can rewrite the sum as:
\[
S_{p_{i_1}\dots p_{i_k}}
=
\sum_{\substack{s'\in (\ZZ/D'\ZZ)^{\times},\\ s'\equiv 1 (\text{mod}~ 6)}}\ds \frac{D'}{3/2}\frac{\Theta_{\mu}\left(D'\tau_{\A_s}\right)}{\Theta\left(\tau_{\A_s}\right)}
\chi_{D'}(\A_s)
\sum_{\substack{s\in (\ZZ/D\ZZ)^{\times},\\ s, s' \equiv 1 (\text{mod}~ 6)\\s\equiv s' (\text{mod}~ D')} } 
\chi_m(\A_s).
\]

In the inner sum we are summing over $s$ modulo $m$ for all $s$ in $\ds (\ZZ/m\ZZ)^{\times}$. Moreover, $\chi_m(\A_s)$ is a nontrivial character as a function of $s$, as  $m^{1/3}\chi_m(\A_s)=(m^{1/3})^{\sigma_{\A_s}}=m^{1/3}$ for all $\A_s$ iff $m^{1/3}\in \QQ[\sqrt{-3}]$. As we are summing a non-trivial character over a group, the sum is $0$. This finishes the proof.

\bigskip

Using the above proposition now it is easy to see:

\begin{cor}\label{S_zero} $S_D=\frac{1}{3c_{3D}}S_{1/6}$ and $S_{1/2}=0$.
\end{cor}

{\it Proof.} As $3c_{3D}S_D=S_{1/6}+1/2S_{1/2}$, if $S_{1/2}=0$ we have $3S_Dc_{3D}=S_{1/6}$. Thus it is enough to show $S_{1/2}=0$. More precisely we will show that $R(z):=\ds \sum_{\substack{r\in (\ZZ/D\ZZ)^{\times}\\ r\equiv 1(6)}} \theta_{r, 1/2}(z)\chi_{\pi}(r)$ equals $0$ for any $z$, in particular for $z=\tau$. Since we showed that $S_{1/2}=D^{-2/3}|R(\tau) |^2$ in Proposition \ref{absolute_value}, we get $S_{1/2}=0$. 

To show $R(z)=0$, note that $\theta_{r, 1/2}(z)=-\theta_{2D-r, 1/2}(z)$, while $\chi_{\pi}(2D-r)=\chi_{\pi}(r)$. As both $r, 2D-r\equiv 1\mod 6$ and $D$ odd, the terms cancel each other out in the sum and we get $R(z)=0$.

\bigskip
Finally, from Proposition \ref{absolute_value} and Corollary \ref{S_zero} we get:

\begin{Prop}\label{absolute_value1} For $\tau=\frac{-b+\sqrt{-3}}{2}$ such that $b^2\equiv -3\mod 12D^2$, we have $S_{D}=\frac{1}{3c_{3D} }S_{1/6}$ and
\[
S_{D}
=
\frac{D^{-2/3}}{3c_{3D}}|\sum\limits_{\substack{s\in (\ZZ/D\ZZ)^{\times} \\ s\equiv 1 (6)}}f_{s, 1/6}\left(\tau\right) 
\chi_{\pi}(s)|^2.
\]
\end{Prop}

\subsection{$S_D$ as a square}
In the following we will rewrite Proposition \ref{absolute_value1} so that we get a square. Define $F_{r, \mu}(z)=f_{r, \mu}(3z)$ and take:
\[
\ds R_{D, \mu}(z)=\sum_{\substack{r \in (\ZZ/D\ZZ)^{\times} \\ r\equiv 1 (6)}} F_{r, \mu}(z)\chi_{\pi}(r).
\]
With this notation, we have showed in Proposition \ref{absolute_value} that $\ds S_{1/6}=|R_{D,1/6}(\tau/3)D^{-1/3}|^2.$ One can show that $R_{D, 1/6}^3(\tau/3)\in K$ and actually $R_{D, 1/6}(\tau/3)$ is really close to being an integer. We will show in this section the following:
\begin{Prop}\label{M_D} For $\sigma(D)$ the number of prime divisors of $D$, we have:
\[
S_D= \frac{(-1)^{\sigma(D)}}{3c_{3D}}T_D^2,
\]
where $T_D=R_{D, 1/6}(\tau/3) \overline{\pi}^{-2/3} \omega^{k_0}$ and $T_D=(-1)^{\sigma(D)}\overline{T_D}$ and thus $T_D$ is real or purely imaginary. 
Here $\omega^{k_0}$ is the unique cube root that makes $T_D$ real or purely imaginary.
\end{Prop}



We are going to show first in Lemma \ref{z_over_3} that $R_{D, 1/6}(\tau)$ and $R_{D, 1/6}(\tau/3)$ differ only by a cubic root of unity $\omega^{k}$, and thus $S_{1/6}=|R_{D, 1/6}(\tau)D^{-1/3}|^2$. In Proposition \ref{square2} we show that $\ds R_{D, 1/6}(\tau)=(-1)^{\sigma(D)}\frac{\overline{\pi}^{2/3}}{\pi^{2/3}}\overline{R_{D, 1/6}(\tau)} \omega^{k'}$. 
Defining $T_D=R_{D, 1/6}(\tau/3) \overline{\pi}^{-2/3} \omega^{k_0}$ for $k_0=k+k'$,  this is equivalent to $T_D=(-1)^{\sigma(D)}\overline{T_D}$ and thus
\[
S_{1/6}=(-1)^{\sigma(D)} T_D^2,
\]
which is the result of Proposition \ref{M_D} above.

\bigskip
\subsubsection{Relating $S_D$ to $R_{D, 1/6}(\tau)$}
We will first show that $S_{1/6}=|R_{D, 1/6}(\tau)D^{-1/3}|^2$  in Lemma \ref{z_over_3}. Define the theta function $\ds \theta^{(r), \mu}(z)=\sum\limits_{n\in \ZZ} e^{\pi i (n-\mu)^2 z} (-1)^n e^{2\pi i nr/D}$ and the ratio $\ds F^{(r), \mu}(z)=\frac{\theta^{(r), \mu}(3z)}{\theta_0(3z)}.$ We introduce this notation, as we will use the transformation mentioned in Lemma \ref{Fourier} in the Appendix:
\begin{equation}\label{FT}  
\theta_{r, 1/6}(3z)=(-1)^r\frac{(-1)^{\frac{D-1}{6}}\omega}{\sqrt{-3}\sqrt{-iz}}(\theta^{(3r), 1/6}(-3/z)-\omega \theta^{(-3r), 1/6}(-3/z)-\omega^2\theta^{(-3r), 1/2}(-3/z)).
\end{equation}

Using also $\theta_{0}(3z)=\frac{1}{\sqrt{-iz}} \theta_0(-3/z)$ and taking the ratio with (\ref{FT}) we get:
\begin{equation}\label{FT1}  
F_{r, 1/6}(z)=(-1)^r\frac{(-1)^{\frac{D-1}{6}}\omega}{\sqrt{-3}}(F^{(3r), 1/6}(-1/z)-\omega F^{(-3r), 1/6}(-1/z)-\omega^2F^{(-3r), 1/2}(-1/z)).
\end{equation}

Then, using (\ref{FT1}), we are ready to show:

\begin{Lem}\label{z_over_3} For $\tau=\frac{-b+\sqrt{-3}}{2}$ such that $b^2\equiv -3\mod 12D^2$, we have: 
\[
R_{D, 1/6}(\tau)=\omega^kR_{D, 1/6}(\tau/3), 
\]
\noindent where $\omega^k$ is a cubic root of unity. Furthermore, this implies $\ds S_{1/6}=|R_{D, 1/6}(\tau)D^{-1/3}|^2. $

\end{Lem}

{\it Proof.} Let $b'\equiv b\mod 4D^2$, and $b'\not\equiv 0\mod 3$. Without loss of generality we can actually pick $b, b'$ such that $(b^2+3)/12D^2$ and $(b'^2+3)/4D^2$ are prime to $3D$. Let $\pi\equiv 1(3)$ be an element of norm $D$ such tht $(\pi)^2$ divides $(\tau)$. Then we can find ideals $\A, \A'$ prime to $3D$ such that:
 \[
 (\sqrt{-3})(\pi)^2\A=\left(\frac{-b+\sqrt{-3}}{2}\right), (\pi)^2\A'=\left(\frac{-b'+\sqrt{-3}}{2}\right).
 \] 
We can write the generators $k_{\A}, k_{\A'} \equiv 1 \mod 3$ of $\A$ and $\A'$, respectively, in the form $k_a=an_a+m_a\frac{-b+\sqrt{-3}}{2}$, $k_{a'}=a'n_{a'}+m_{a'}\frac{-b+\sqrt{-3}}{2}$, where $m_a, m_a'\equiv 0 \mod 3$, and $n_a, n_a'\equiv 1\mod 3$. Let $\tau_{\A}=\frac{-b+\sqrt{-3}}{2a}, \tau_{\A'}=\frac{-b'+\sqrt{-3}}{2a'}$.

We are going to show that:
  
  	\begin{enumerate}[(i)]
	
		\item $R_{D, 1/6}(\tau_{\A'})=R_{D, 1/6}(\tau_{\A}/3)$
		
		\item $R_{D, 1/6}(\tau_{\A}/3)=\chi_{\overline{\pi}}(n_a)R_{D, 1/6}(\tau/3)$
		
		These two relations will imply: 
		
		\item $R_{D, 1/6}(\tau)=\chi_{\overline{\pi}}(n_a)\chi_{\pi}(n_{a'})R_{D, 1/6}(\tau/3)$.
	\end{enumerate}

  In order to show $R_{D, 1/6}(\tau_{\A'})=R_{D, 1/6}(\tau_{\A}/3)$, note that it is enough to show that $F_{r, 1/6}(\tau_{\A'})=F_{r, 1/6}(\tau_{\A}/3)$. We have $-1/(\tau_A/3)=-\overline{\tau}/D^2$ and $-1/\tau_{\A'}=-\overline{\tau}/D^2$ as well and we will use (\ref{FT1}) for both $\tau_{\A}/3$ and $\tau_{\A'}/3$. First for $z=\tau_{\A}/3$ we get:
\[
F_{r, 1/6}(\tau_{\A}/3)=(-1)^r\frac{(-1)^{\frac{D-1}{6}}\omega}{\sqrt{-3}}(F^{(3r), 1/6}(-\overline{\tau}/D^2)-\omega F^{(-3r),1/6}(-\overline{\tau}/D^2)- \omega^2F^{(3r), 1/2}(-\overline{\tau}/D^2)).
\]
Applying (\ref{FT1}) also for $\tau_{\A'}$ we get similarly 
\[
F_{r, 1/6}(\tau_{\A'})=(-1)^r\frac{(-1)^{\frac{D-1}{6}}\omega}{\sqrt{-3}}(F^{(3r), 1/6}(-\overline{\tau'}/D^2)-\omega F^{(-3r),1/6}(-\overline{\tau'}/D^2)- \omega^2F^{(3r), 1/2}(-\overline{\tau'}/D^2)),
\] 
where $\tau'=\frac{-b'+\sqrt{-3}}{2}$. 

Finally, note that $\ds F^{(s), 1/6}(z+8D^2)=F^{(s), 1/6}(z)$, thus since $b\equiv b'\mod 8D^2$ we also have $F^{(s), 1/6}(-\overline{\tau'}/D^2)=F^{(s), 1/6}(-\overline{\tau}/D^2)$ for $s=\pm 3r$. Similarly $F^{(-3r), 1/2}(-\overline{\tau'}/D^2)=F^{(-3r), 1/2}(-\overline{\tau}/D^2)$, thus $F_{r, 1/6}(\tau_{\A}/3)=F_{r, 1/6}(\tau_{\A'})$ as claimed.

 \bigskip
To show (ii), as $f_r(z)=F_r(z/3)$, note that from Lemma \ref{theta_galois} we have $F_{r, 1/6}(\tau/3)^{\sigma_{\A}^{-1}}=F_{r, 1/6}(\tau_{\A}/3)=F_{n'_ar, 1/6}(\tau/3)$ where $n_a\equiv n'_a \mod 3D$ and $n'_a$ odd. This further implies that $R_{D, 1/6}(\tau_{\A}/3)=\chi_{\overline{\pi}}(n_a)R_{D, 1/6}(\tau/3)$.

  \bigskip
   To show (iii), note that we are in the conditions of Lemma \ref{galois_conj} from Section \ref{Shimura}, as $F_{r, 1/6}$ is a modular function of level $18D^2$. Then $F_{r, 1/6}(\tau)^{\sigma_{\A'}^{-1}}=F_{r, 1/6}(\tau_{\A'})$ and thus we get $R_{D, 1/6}(\tau_{\A'})=(R_{D, 1/6}(\tau))^{\sigma_{\A'}^{-1}}.$

We can rewrite this as $(R_{D, 1/6}(\tau_{\A'}))^{\sigma_{\A'}}=R_{D, 1/6}(\tau)$ and using $(i)$, we get $R_{D, 1/6}(\tau)=R_{D, 1/6}(\tau_{\A}/3)^{\sigma_{\A'}}$.  From $(ii)$, this is $R_{D, 1/6}(\tau)=\chi_{\overline{\pi}}(n_a)R_{D, 1/6}(\tau/3)^{\sigma_{\A'}}$. Using Lemma \ref{theta_galois}, we have $F_{r, 1/6}(\tau/3)^{\sigma_{\A'}^{-1}}=F_{n_{a'}r, 1/6}(\tau/3)$, thus $R_{D, 1/6}(\tau/3)^{\sigma_{\A'}}=\chi_{\pi}(n_{a'})R_{D, 1/6}(\tau/3)$. Finally this implies $R_{D, 1/6}(\tau)=\chi_{\overline{\pi}}(n_a)\chi_{\pi}(n_{a'})R_{D, 1/6}(\tau/3)$ and we take $\omega^k=\chi_{\overline{\pi}}(n_a)\chi_{\pi}(n_{a'})$ to get the result.

  \bigskip

\subsubsection{Relating $R_{D, 1/6}(\tau)$ to its complex conjugate}

Now we want to show that $R_{D, 1/6}(\tau)$ equals $\overline{R_{D, 1/6}(\tau)}$ up to a nice factor. As before we let $\tau=\frac{-b+\sqrt{-3}}{2}$ and $\pi$ such that $(\pi)^2$ divides $(\tau)$. We will show:

\begin{Prop}\label{square2} For some cubic root of unity $\omega^{k'}$, we have:
\[
R_{D, 1/6}(\tau)=(-1)^{\sigma(D)}\omega^{k'}\frac{\overline{\pi}^{2/3}}{\pi^{2/3}}\overline{R_{D, 1/6}(\tau)}.
\]
Using the notation $T_D=R_{D, 1/6}(\tau) \overline{\pi}^{-2/3}\omega^{k'}$ this is equivalent to $T_D=(-1)^{\sigma(D)}\overline{T_D}$.
\end{Prop}

Note that we can think of $\omega^{k'}$ as the unique root of unity which makes $R_{D, 1/6}(\tau)\omega^{k'}\overline{\pi}^{-2/3}$ either real or purely imaginary. We actually give a formula for $\omega^{k'}$ in the proof of Proposition \ref{square2}.

\bigskip





We first define the linear combination:
\[
\ds R^{(D), \mu}(z)=\sum_{\substack{r \in (\ZZ/D\ZZ)^{\times} \\ r\equiv 1 (6)}} F^{(r), \mu}(z)\chi_{\overline{\pi}}(r).
\] Note that we use $\chi_{\pi}$, unlike in $R_{D, \mu}$. 

We choose $b\equiv b' \mod 4D^2$ such that $3\nmid b'$ and we can find$\A'$ as in the proof of Lemma \ref{z_over_3} such that $\A'(\pi)^2=\left(\frac{-b'+\sqrt{-3}}{2}\right)$. Then from the transformation (\ref{FT1}) we have $F_{r, 1/6}(\tau_{\A'})=(-1)^r\frac{(-1)^{\frac{D-1}{6}}\omega}{\sqrt{-3}}(F^{(3r), 1/6}(-\overline{\tau}/D^2)-\omega F^{(-3r), 1/6}(-\overline{\tau}/D^2)-\omega^2F^{(-3r), 1/2}(-\overline{\tau}/D^2))$. Writing the full linear combination for $r\in (\ZZ/D\ZZ)^{\times}$, $r\equiv 1(6)$ and multiplying by $\chi_{\pi}(3)$, we get:
\begin{equation}\label{R_sum}
(-1)^{\frac{D+1}{2}}\chi_{\pi}(3)R_{D, 1/6}(\tau_{\A'})=\overline{R^{(D), 1/6}(\tau/D^2)}-\omega^2\overline{R^{(D), 1/2}(\tau/D^2)}/\sqrt{-3}.
\end{equation}

Note that above we related $R_{D, 1/6}(\tau)$ to $\overline{R^{(D), \mu}}$ for $\mu\in\{1/2, 1/6\}$. In order to show Proposition \ref{square2} we also want to relate $R^{(D), \mu}$ back to $R_{D, \mu}$, and we do that in the lemma below:

\begin{Lem}\label{unwind}  $\ds R^{(D), \mu}(\tau/D^2)
=
(-1)^{(D+1)/2}\frac{G(\chi_{\overline{\pi}})}{\overline{\pi}}R_{D, \mu}(\tau).$

\end{Lem}	
	
{\it Proof.} Recall $\ds R^{(D), \mu}(z/D^2)=\sum\limits_{r\in (\ZZ/D\ZZ)^{\times}}F^{(r), \mu}(z/D^2)\chi_{\overline{\pi}}(r)$, where $\ds F^{(r), \mu}(z/D^2)=\frac{\theta^{(r), \mu}(3z/D^2)}{\theta_0(3z/D^2)}$. We show first that, for $r$ odd, we can rewrite the terms $\theta^{(r), \mu}(3z/D^2)$ as: 
\begin{equation}\label{theta_rr}
\theta^{(r), \mu}(3z/D^2)=-\sum_{\substack{s\in (\ZZ/D\ZZ)\\ s\equiv 1(6)}}\theta_{s, \mu}(3z) e^{2\pi i rs/D}.
\end{equation}

From the definition, we have $\ds \theta^{(r), \mu}(3z)=\sum_{n\in \ZZ}  e^{\pi i (n-D\mu)^2 3z} (-1)^n e^{2\pi i nr/D }$. Choosing as before $s \in \ZZ/D\ZZ$ such that $s\equiv 1 \mod 6$, we sum over all $n$ modulo $D$:
\[
\theta^{(r), \mu}(3z)=\sum_{\substack{s\in \ZZ/D\ZZ\\ s\equiv 1(6) }}\sum_{n\in \ZZ}  e^{\pi i (Dn+s-D\mu)^2 3z} (-1)^{Dn+s} e^{2\pi i (Dn+s)r/D }.
\]
We can rewrite this as $\ds \theta^{(r), \mu}(3z)=-\sum\limits_{\substack{s\in \ZZ/D\ZZ\\ s\equiv 1(6) }} \theta_{s, \mu}(D^23z) e^{2\pi i sr/D }$ and changing $z\ra z/3D^2$ we get (\ref{theta_rr}). 

Plugging in $z=\tau/D^2$ in (\ref{theta_rr}) and dividing by $\theta_0(3\tau)$ we have further $\ds \frac{\theta^{(r), \mu}(3\tau/D^2)}{\theta_0(3\tau)}=-\sum_{\substack{s\in \ZZ/D\ZZ\\ s\equiv 1(6)}}F_{s, \mu}(\tau) e^{2\pi i rs/D}$. Moreover from Lemma \ref{eta_lem} in the Appendix $\ds \frac{\theta_0(3\tau/D^2)}{\theta_0(3\tau)}=\frac{\eta(\tau/D^2)}{\eta(\tau)}=(-1)^{(D-1)/6}\overline{\pi}$, thus we get:
\begin{equation}
\ds F^{(r), \mu}(\tau/D^2)
=
\frac{(-1)^{(D+1)/2}}{\overline{\pi}} \sum_{\substack{s\in \ZZ/D\ZZ\\ s\equiv 1(6) }}F_{s, \mu}(\tau) e^{2\pi i rs/D}.
\end{equation}

Going back to the linear combination, we get $\ds R^{(D), \mu}(\tau/D^2)=-\frac{1}{\overline{\pi}}\sum_{\substack{r\in (\ZZ/D\ZZ)^{\times}\\ r\equiv 1(6)}}\sum_{\substack{s\in \ZZ/D\ZZ\\ s\equiv 1(6) }}F_{s, \mu}\left(\tau\right)e^{2\pi i sr/D}\chi_{\overline{\pi}}(r)$. We switch the two sums and get:
\[
R^{(D), \mu}(z/D^2)=\frac{(-1)^{(D+1)/2}}{\overline{\pi}}\sum_{\substack{(s\in \ZZ/D\ZZ)^{\times} \\s\equiv 1(6)}}F_{s, \mu}\left(\tau\right)\sum_{r\in (\ZZ/D\ZZ)^{\times}} e^{2\pi i sr/D}\chi_{\overline{\pi}}(r).
\]

Note that if $\gcd(s, D)=D_1>1$, then the inner sum equals $0$. This is easily seen by writing $s=D_1s_1$, $D_2=D/D_1$, and rewriting 
$\ds \sum_{r\in (\ZZ/D\ZZ)^{\times}} e^{2\pi i sr/D}\chi_{\overline{\pi}}(r)=\chi_{\overline{\pi}}(s_1)\sum_{r\in (\ZZ/D\ZZ)^{\times}} e^{2\pi i r/D_1}\chi_{\overline{\pi}}(r)=\sum_{r\in(\ZZ/D_2\ZZ)^{\times} }  \chi_{\overline{\pi}_2}(r) G(\chi_{\overline{\pi}_1})=0,$ where $\pi = \pi_1\pi_2$ and $\Nm(\pi_i)=D_i$, for $i=1,2$. Thus we are left in our sum only with $s$ prime to $D$, and we can rewrite:
\[
R^{(D), \mu}(z/D^2)=\frac{(-1)^{(D+1)/2}}{\overline{\pi}}\sum_{(s\in \ZZ/D\ZZ)^{\times}}F_{s, \mu}(\tau) \chi_{\pi}(s)\sum_{r\in (\ZZ/D\ZZ)^{\times}} e^{2\pi i sr/D}\chi_{\overline{\pi}}(rs).
\]

This is exactly $\ds R^{(D), \mu}(\tau/D^2)=
\frac{(-1)^{(D+1)/2}}{\overline{\pi}}\sum_{\substack{(s\in \ZZ/D\ZZ)^{\times}\\ s\equiv 1(6)}}F_{s, \mu}\left(\tau\right)\chi_{\pi}(s) G(\chi_{\overline{\pi}})= (-1)^{(D+1)/2}\frac{G(\chi_{\overline{\pi}})}{\overline{\pi}} R_{D, \mu}(\tau),$ and thus we got the result of our lemma.

\bigskip

{\bf Proof of Proposition \ref{square2}:} Recall from equation (\ref{R_sum}), we have $\ds (-1)^{\frac{D-1}{2}}\chi_{\pi}(3)R_{D, 1/6}(\tau_{\A'})=\overline{R^{(D), 1/6}(\tau/D^2)}-\omega^2\overline{R^{(D), 1/2}(\tau/D^2)}/\sqrt{-3}$.
Rewriting the RHS using Lemma \ref{unwind}, then $\ds \chi_{\pi}(3) R_{D, 1/6}(\tau_{\A'})=\frac{G(\chi_{\pi})}{\pi}(\overline{R_{D, 1/6}(\tau)}-\omega^2\overline{R_{D, 1/2}(\tau)}/\sqrt{-3}).$ Finally, since we noted that $\ds R_{D, 1/2}(z)=0$ in the proof of Lemma \ref{S_zero}, for any $z$, then we get $\ds R_{D, 1/6}(\tau_{\A'})=\chi_{\overline{\pi}}(3)\frac{G(\chi_{\pi})}{\pi}\overline{R_{D, 1/6}(\tau)}.$


Now using the details of the proof of Lemma \ref{z_over_3}, recall that $R_{D, 1/6}(\tau_{\A'})=R_{D, 1/6}(\tau_{\A}/3)=\chi_{\overline{\pi}}(n_{a})R_{D, 1/6}(\tau/3)=\chi_{\overline{\pi}}(n_{a'})R_{D, 1/6}(\tau)$, thus 
\begin{equation}\label{R_D}
R_{D, 1/6}(\tau)=\chi_{\pi}(n_{a'})\chi_{\overline{\pi}}(3)\frac{G(\chi_{\pi})}{\pi}\overline{R_{D, 1/6}(\tau)}.
\end{equation}

To actually compute the term on the RHS, we recall a few facts about cubic Gauss sums. We can write $\pi =\prod\limits_{p_i |D} \pi_i$, where $\pi_i$ is a generator of norm $p_i$ with $\pi_i \equiv 1 (3)$. Then:
\[
G(\chi_{\pi})=\prod_{p_i|D} \chi_{\pi/\pi_i}(\pi_i) G(\chi_{\pi_i}).
\]

Moreover, we can actually compute each $G(\chi_{\pi_i})$ up to a cubic root of unity. Following \cite{IR}, we have $G(\chi_{\pi_i})^3=J(\chi_{\pi_i}, \chi_{\pi_i})$, where $J(\chi_{\pi_i}, \chi_{\pi_i})$ is the Jacobi sum for the character $\chi_{\pi_i}$. Moreover, we can compute $J(\chi_{\pi_i}, \chi_{\pi_i})=-p\overline{\pi}_i$ for $\pi_i\equiv 1\mod 3$ (also see \cite{IR}). Thus we get $G(\chi_{\pi_i})=-\overline{\pi_i}^{2/3}\pi_i^{1/3} \omega^{k_i}$ for some $k_i\in\{0, 1, 2\}$, 

Then $\ds \frac{G(\chi_{\pi})}{\pi}=(-1)^{\sigma(D)} \omega^{k_D} \frac{\overline{\pi}^{2/3}}{\pi^{2/3}}$, where $k_D=\sum k_i$, which together with (\ref{R_D}) gives us Proposition \ref{square2} for $\omega^{k'}=\chi_{\pi}(n_{a'})\chi_{\overline{\pi}}(3)\omega^{k_D}$.

\bigskip

From Lemma \ref{z_over_3} and Proposition \ref{square2}, we get $T_D=R_{D, 1/6}(\tau/3)\overline{\pi}^{-2/3}\omega^{k_0}$ for $k_0=k+k'$, and thus:

\begin{cor}\label{square_minus} $S_D=(-1)^{\sigma(D)} T_D^2$.

\end{cor}




\subsection{Invariance under the Galois action}
Define 
\[
M_D= R_{D, 1/6}(\tau/3) \overline{\pi}^{-2/3}.
\]
We will write below $M_D$ as a trace.

 \begin{Lem}\label{trace_H} $M_{D}\in K$ and we can write it as a trace:
\[
M_{D}=\frac{1}{\varphi(D)}\Tr_{H_{\OO}/K} (f_{1, 1/6}(\tau)\overline{\pi}^{-2/3}),
\]
where $H_{\OO}$ is the ray class field of modulus $3D$, $\varphi$ is Euler's totient function and $f_{1, 1/6}(\tau)=\theta_{1, 1/6}(\tau)/\theta_0(\tau)$
\end{Lem}

\bigskip

Recall we defined $T_D=R_{D, 1/6}(\tau) \overline{\pi}^{-2/3}\omega^{k'}$ and from Lemma \ref{z_over_3} we have $T_D=R_{D, 1/6}(\tau/3) \overline{\pi}^{-2/3}\omega^{k_0}$ for $k_0=k+k'$. As $T_D=(-1)^{\sigma}\overline{T_D}$ from Proposition \ref{square2}, we get immediately from Lemma \ref{trace_H}:


\begin{cor}\label{rational} $T_D\in \QQ$ when $\sigma(D)$ even and $T_D/\sqrt{-3}\in \QQ$ when $\sigma(D)$ odd and we have the formula:
\[
T_{D}=\frac{\omega^{k_0}}{\varphi(D)}\Tr_{H_{\OO}/K} (f_{1, 1/6}(\tau)\overline{\pi}^{-2/3}).
\]
\end{cor}
\bigskip

{\bf Proof of Lemma \ref{trace_H}}:  We can write explicitly $M_D=\sum\limits_{r\in (\ZZ/D\ZZ)^{\times}} f_r(\tau) \chi_{\pi}(r) \overline{\pi}^{-2/3}$. For an ideal $\A=[a, \frac{-b+\sqrt{-3}}{2}]_{\ZZ}$ with generator $k_a=\left(n_a a+m_a\frac{-b+\sqrt{-3}}{2}\right)$ with $6|m_a$, $n_a\equiv 1 (6)$, we are going to have the Galois transformation:
\[
(f_{r, 1/6}(\tau) \chi_{\pi}(r) \overline{\pi}^{-2/3})^{\sigma_{\A}^{-1}}=f_{n_ar, 1/6}(\tau) \chi_{\pi}(n_ar) \overline{\pi}^{-2/3}.
\]
To show this, note  that from Lemma \ref{theta_galois} we have $f_{r, 1/6}(\tau)^{\sigma_{\A}^{-1}}=f_{rn_{a}, 1/6}(\tau)$. We compute $(\overline{\pi}^{1/3})^{\sigma_{\A}^{-1}}=\chi_{\overline{\pi}}(\A) \overline{\pi}^{1/3}$. Furthermore,  $\chi_{\overline{\pi}}(\A)=\overline{\left(\frac{\overline{\pi}}{\A}\right)_3}=\left(\frac{\pi}{\overline{\A}}\right)_3=\left(\frac{\pi}{n_aa+m_ab}\right)_3$ and $n_a(n_aa+m_ab)^2 \equiv a \mod \pi$, so we have $\left(\frac{\pi}{n_aa+m_ab}\right)_3=\left(\frac{\pi}{n_a^{-1}}\right)_3=\overline{\left(\frac{\pi}{n_a}\right)_3}=\chi_{\pi}(n_a)$.

Moreover, taking the ideals $\A_{r}^{\circ}=\left(1+b^{*}(1-r^*)\frac{-b+\sqrt{-3}}{2}\right)$, where $b^*\equiv b^{-1} \mod D$ and $r^{*}\equiv r^{-1}\mod D$, we have $\Nm\A_r=a_{\A_r^{\circ}}\equiv r^{-1} \mod 3D $ and $n_{\A_r^{\circ}}\equiv r \mod 3D$, and then:
\[
M_D=\sum\limits_{\substack{r\in (\ZZ/D\ZZ)^{\times} \\ r\equiv 1(6)}} (f_1(\tau)\overline{\pi}^{-2/3})^{\sigma_{\A_r^{\circ}}^{-1}}.
\]

Define the group $G_0=\{[\A_{r}^{\circ}], r\in (\ZZ/D\ZZ)^{\times}\}$. It is a subgroup of $\Gal(H_{\OO}/K)$ and $G_0$ isomorphic to $(\ZZ/D\ZZ)^{\times}$. We define the fixed field of $G_0$ in $\Gal(H_{\OO}/K)$ to be $H_0=\{h\in H_{\OO}: \sigma(h)=h, \forall \sigma \in G_0\}$ and from Galois theory this implies $\Gal(H_{\OO}/H_0)\cong G_0$. Then we can rewrite the relation above as 
\[
\ds M_{D}=\Tr_{H_{\OO}/H_0} (f_1(\tau)\overline{\pi}^{-2/3}).
\]
Moreover, if we take the trace further to $K$, we get $\Tr_{H_{\OO}/K} (f_1(\tau)\overline{\pi}^{-2/3})=\frac{\#\Gal(H_{\OO}/K)}{\#\Gal(H_{\OO}/H_0)}\Tr_{H_{\OO}/H_0} (f_1(\tau)\overline{\pi}^{-2/3})=\prod\limits_{p|D}(p-1) \Tr_{H_{\OO}/H_0} (f_1(\tau)\overline{\pi}^{-2/3})$.

\begin{rmk} Using the notation from the proof of Lemma \ref{trace_H}, one can actually show similarly that $\kappa=R_{D, 1/6}D^{-1/3}$ equals:
\[
\kappa=\Tr_{H_{\OO}/H_0} (f_1(\tau)D^{-1/3}) 
\]
and $\kappa^3\in K$. However, $\kappa \notin K$.

\end{rmk}

\bigskip

\subsection{Integrality}\label{int}

In Section \ref{Zeta_functions} we have showed that $S_D\in \QQ$. We will show below that $S_{1/6}\in \ZZ$, thus $3c_{3D} S_D \in \ZZ$.

 Recall that $S_{1/6}= 3c_{3D}S_D=\Tr_{H_{3D}/K}\frac{\Theta(D\omega)}{\Theta(\omega)} D^{1/3}$. Note that it is enough to show that $D^{1/3}\Theta(D\omega)/\Theta(\omega)$ is an algebraic integer, as its trace would be a rational number as well as an algebraic integer, thus an integer. Moreover, it is enough to show that $\Theta(D\omega)/\Theta(\omega)$ is an algebraic integer. 

We will use the following Lemma:

\begin{Lem}\label{algebraic} Let $f(z)$ be a modular function for $\Gamma(N)$ such that for all $\gamma\in \SL_2(\ZZ)$ we have $f\circ \gamma$ holomorphic on the upper half plane $\HH$ and $f$ has Fourier coefficients at $\infty$ that are algebraic integers. Then, for $\tau$ a CM point, $f(\tau)$ is an algebraic integer 
\end{Lem}

\begin{proof} Define the polynomial $P(X)=\prod\limits_{\gamma \in \Gamma(N)\setminus\SL_2(\ZZ)}(X-f(\gamma z))$ and note that all its coefficients are modular functions that are invariant under $\SL_2(\ZZ)$. Moreover they are holomorphic functions on the upper half plane and have Fourier coefficients at $\infty$ that are algebraic integers. Then we can write each coefficient $c(z)$ as a polynomial in $j(z)$ with coefficients that are algebraic integers.

Then, for $z=\tau$, since $j(\tau)$ is an algebraic integer, we get that $f(\tau)$ is the root of a polynomial with coefficients that are algebraic integers, and thus $f(\tau)$ is an algebraic integer as well.
 \end{proof}

First we will show that $2\frac{\Theta(D\omega)}{\Theta(\omega)}$ is an algebraic integer. We have showed that $\ds \frac{3}{2}\Theta(\omega)
=
\sqrt[4]{3}|\theta_0(\tau)|^2$. Then $\ds 2\frac{\Theta(D\omega)}{\Theta(\omega)}=3^{3/4}\frac{\Theta(D\omega)}{|\theta_0(\tau)|^2.}=3^{3/4}e^{-2\pi i/24}\frac{\Theta(-D\overline{\tau})}{\theta_0(-\overline{\tau})^2.}$

Since $e^{-2\pi i/24}3^{3/4}$ is an algebraic integer, it is enough to show that $\frac{\Theta(-D\overline{\tau})}{\theta_0(-\overline{\tau})^2}$ is one as well. Recall that $\theta_0(z)=\eta(z/3)$ and take $\ds f_0(z)=\frac{\Theta(Dz)}{\eta(z/3)^2}$. 

Note that:
\begin{itemize}
	\item $f_0$ is a modular function for $\Gamma(36D)$;
	\item $f_0(\gamma z)$ is holomorphic on $\HH$ for all $\gamma \in \SL_2(\ZZ)$;
	\item $f_0(\gamma z)$ has Fourier coefficients that are algebraic integers in its Fourier expansion at $\infty$ for all $\gamma \in \SL_2(\ZZ)$.
 
\end{itemize}

These properties can be checked using the properties of $\Theta_K$ from the Appendix as well as the properties of $\eta(z)$. Note that we are in the conditions of Lemma \ref{algebraic}, thus $f(\tau)$ is an algebraic integer. This implies that $2\frac{\Theta(D\omega)}{\Theta(\omega)}$ is an algebraic integer, hence $2S_{1/6}$ is an integer.

\bigskip

Now we will show that $DS_{1/6}$ is an integer as well by showing that $D\frac{2}{3}\frac{\Theta_{1/6}(D\omega)}{\Theta(\omega)}$ is an algebraic integer. Using Lemma \ref{sum_rD} for $\mu\in\{-1/2, -1/6\}$, $\nu=1/2$, $a=1$, $z=D\frac{-3+\sqrt{-3}}{2}$, we can rewrite:
\[
\Theta_{1/6}(D\tau)
=
\frac{\sqrt[4]{3}}{D}\sum_{r=0}^{D-1} |\theta_{r,1/6} (\tau)|^2, \ \tau=\frac{-3+\sqrt{-3}}{2}.
\]


Taking the quotient by $\frac{3}{2}\Theta(\tau)
=
\sqrt[4]{3}|\theta_0(\tau)|^2$, we get:
\[
\frac{2}{3}\frac{\Theta_{1/6}(D\omega)}{\Theta(\omega)}
=
\frac{1}{D}\sum_{r=0}^{D-1} \left|\frac{\theta_{r, 1/6}(\tau)}{\theta_0(\tau)}\right|^2.
\]

Recall $\ds f_{r, \mu}(z)=\frac{\theta_{r, \mu}(z)}{\theta_0(z)}$ for $\mu\in \{1/2, 1/6\}$ and note that:

\begin{itemize}
	\item $f_{r, \mu}$ is a modular function for $\Gamma(18D^2)$;
	\item $f_{r, \mu}(\gamma z)$ is holomorphic on $\HH$ for all $\gamma \in \SL_2(\ZZ)$;
	\item $f_{r, \mu}(\gamma z)$ has Fourier coefficients that are algebraic integers in its Fourier expansion at $\infty$ for all $\gamma \in \SL_2(\ZZ)$.
\end{itemize}

To show the last property we can use the automorphic definitions of both $\theta_{r, \mu}$ and $\theta_0$. We get a Fourier expansion with coefficients in $\OO_K[\zeta_{24}, \zeta_{D^2}]$. Thus we are in the conditions of Lemma \ref{algebraic}, hence $f_{r, \mu}(\tau)$ is an algebraic integer. This implies that $D\frac{\Theta_{1/6}(D\omega)}{\Theta(\omega)}$ is an algebraic integer, and thus the trace $DS_{1/6}=D\Tr_{H_{3D}/K}\frac{\Theta_{1/6}(D\omega)}{\Theta(\omega)}D^{1/3}$, which is a rational number, is indeed an integer. Since we already showed that $2S_{1/6}$ is an integer, we get $S_{1/6}\in \ZZ$ when $D$ is odd.


\bigskip 

If we note that $|\theta_0(\tau)|^2=|\eta(\tau/3)|^2=\sqrt{3}|\eta(\tau)|^2$, then we can also show similarly that $\frac{\theta_{r, \mu}(\tau)}{\eta(\tau)}$ is an algebraic integer, and thus $DS_{1/6}/\sqrt{3}$ is an algebraic integer. Since $S_{1/6}\in \ZZ$, this implies $3$ divides $S_{1/6}$.

Finally, since $S_{1/6}=(-1)^{\sigma_D}T_D^2$ and $S_{1/6}$ is an integer, from Corollary \ref{rational} we get:

\begin{cor} $T_D/3\in \ZZ$ when $\sigma(D)$ even and $T_D/\sqrt{-3}\in \ZZ$ when $\sigma(D)$ odd.
\end{cor}

Note that this implies that $\omega^{k_0}$ is the unique choice for a cube root of unity such that $T_D$ or $T_D/\sqrt{-3}$ is an integer.

\subsection{Case of $D$ not square free}\label{not_square_free}
We present below the case of $D$ not square free. We write $D=D_1D_2^2$ such that $D_1D_2$ square free. In this case we use the formula (\ref{S_D}):
\begin{equation*}
S_D=\sum_{\A \in \Cl(\OO_{3D_0})}\frac{\Theta_K(D_0\tau_{\A})}{\Theta_K(\tau_{\A})}\chi_{D}(\A)D^{1/3}.
\end{equation*}
for $D_0=D_1D_2$. All the details of the proof for the square-free case will follow through and we only briefly mention the steps. We apply the factorization formula and obtain the factorization from Corollary \ref{ratio}:
\begin{equation}
\frac{\Theta_{\mu}(D_0\tau_{\A})}{\Theta\left(\tau_{\A}\right)}
=
\frac{3/2}{D_0}\sum\limits_{r\in \ZZ/D_0\ZZ}
f_{ar, \mu}\left(\tau_{\A^2\A_1}\right) \overline{f_{r, \mu}\left(\tau_{\A_1}\right)},
\end{equation}
We use this to show similarly to the proof of Proposition \ref{absolute_value1} that: 
\[
S_{D}
=
\frac{D_2^{1/3}D_1^{-2/3}}{3c_{3D}}|\sum\limits_{\substack{s\in (\ZZ/D_0\ZZ)^{\times} \\ s\equiv 1 (6)}}f_{s, 1/6}\left(\tau\right) 
\chi_D(\A_s)|^2,
\]
where $\A_s=[a_s, \frac{-b+\sqrt{-3}}{2}]$ with $\Nm(a_s)\equiv s \mod D_1D_2$. Take $\pi_1$ a generator of $D_1$ and $\pi_2$ is the generator of $D_2$ such that $\left(\frac{-b+\sqrt{-3}}{2}\right)$ is divisible by $(\pi_1\pi_2)^2$. Then $\chi_D(\A_s)=\chi_{\pi_1}(s)\overline{\chi_{\pi_2}(s)}$.

The main difference is when we compute the complex conjugate of
\[
R_{D, 1/6}(\tau)=\sum_{\substack{r\in (\ZZ/D_0\ZZ)^{\times}\\ r\equiv 1(6)}} f_{r, 1/6}(\tau) \chi_{\pi_1}(s)\overline{\chi_{\pi_2}(s)},
\]
as we get $R_{D, 1/6}(\tau)=\frac{G(\chi_{\pi_1}\chi_{\overline{\pi_2}})}{\pi_1\pi_2}\overline{R_{D, 1/6}(\tau)}$ and this equals $R_{D, 1/6}(\tau)=(-1)^{\sigma(D)}\omega^{k'}\frac{\overline{\pi_1}^{1/3}\overline{\pi_2}^{2/3}}{\pi_1^{1/3}\pi_2^{1/3}}\overline{R_{D, 1/6}(\tau)}$ for a cubic root of unity $\omega^{k'}$. Then we can rewrite:
\[
S_{D}
=
\frac{(-1)^{\sigma(D)}}{3c_{3D}}T_D^2,
\]
where $T_D=\sum\limits_{\substack{s\in (\ZZ/D_1D_2\ZZ)^{\times} \\ s\equiv 1 (6)}}f_{s, 1/6} \left(\tau\right) 
\chi_{\pi_1}(s)\overline{\chi_{\pi_2}(s)} \overline{\pi_1}^{-2/3} \pi_2^{1/3} \omega^{k_0}$, and we can show that this is the trace:
\[
T_D=\frac{1}{\varphi(D_0)}\Tr_{H_{\OO}/K} f_1(\tau)\overline{\pi_1}^{-2/3} \pi_2^{1/3}\omega^{k_0},
\]
where $H_{\OO}$ is the ray class field for the modulus $D_0$ and $\omega^{k_0}$ is a cubic root of unity. Moreover, we can further show as in Section \ref{int} that $T_D/3 \in \ZZ$ when $\sigma(D)$ even and $T_D/\sqrt{-3} \in \ZZ$ when $\sigma(D)$ odd.

Finally, we have $S_D\neq 0$ for $D$ split only for $D\equiv 1 \mod 9$ and in this case the Tamagawa number $c_{3D}$ equals $3^{1+\sigma(D)}$, thus $S_D$ is an integer square up to an even power of $3$ and it equals:
\[
S_{D}
=
\left(\frac{T_D}{(\sqrt{-3})^{2+\sigma(D)}}\right)^2.
\]


\section{Shimura reciprocity law}\label{Shimura}



We present below some background on Shimura's reciprocity law following the exposition of Stevenhagen \cite{St}. For more details also see Gee \cite{Gee}.

 Let $\F$ be the field of modular functions over $\QQ$. From CM theory (see for example \cite{St}), it is known that if $\tau\in K\cap \HH$ and $f\in \F$, then we have $f(\tau)\in K^{ab}$, where $K^{ab}$ is the maximal abelian extension of $K$. Shimura's reciprocity law gives us a way to compute the Galois conjugates $f(\tau)^{\sigma}$ of $f(\tau)$ when acting with $\sigma \in \Gal(K^{ab}/K)$.
We recall that $\F=\bigcup_{N\geq 1} \F_N$, where $\F_N$ is the space of modular functions of level $N$. Moreover, $\F_N$ is the function field of the modular curve $X(N)=\Gamma(N)\setminus \HH^*$ over $\QQ(\zeta_N)$, where $\zeta_N=e^{2\pi i/N}$ and $\HH^*=\HH\cup \mathbb{P}^1(\QQ)$. We can compute explicitly $\F_N=\QQ(j, j_N)$, where $j$ is the $j$-invariant and $j_N(z)=j(Nz)$. In particular, we have $\F_1=\QQ(j)$.

When working over $\QQ$, one has an isomorphism $\Gal(\F_{N}/\F_{1})\cong \GL_2(\ZZ/N\ZZ)/\{\pm 1\}.$ More precisely, if we denote by $g_{\sigma}$ the Galois action corresponding to the matrix $g \in  \GL_2(\ZZ/N\ZZ)$ under the isomorphism above, it is enough to define the Galois action for $\SL_2(\ZZ/N\ZZ)$ and for $G_N=\Big\{\left(\begin{smallmatrix} 1& 0 \\ 0 & d\end{smallmatrix}\right), d\in (\ZZ/N\ZZ)^{\times}\}$. We state explicitly the two actions below:

\begin{itemize}

\item \textbf{Action of $\alpha \in \SL_2(\ZZ/N\ZZ)$ on $\F_N$.} We have $(f(\tau))^{\sigma_{\alpha}}=f^{\alpha}(\tau):=f(\alpha \tau),$ where $\alpha$ is acting on the upper half plane via fractional linear transformations.

\item \textbf{Action of $\left(\begin{smallmatrix} 1& 0 \\ 0 & d\end{smallmatrix}\right)\in (\ZZ/N\ZZ)^{\times}$ on $\F_N$.} Note that for $f\in \F_N$ we have a Fourier expansion $f(z)=\sum\limits_{n\geq 0} a_n q^{n/N}$ with coefficients $a_n\in \QQ(\zeta_N)$, $q=e^{2\pi i z}$. If we denote $u_d=\left(\begin{smallmatrix} 1& 0 \\ 0 & d\end{smallmatrix}\right)$, then the action of $\sigma_{u_d}$ is given by $(f(\tau))^{\sigma_{u_d}}=f^{u_d}(\tau):=\sum\limits_{n \geq 0} a_n^{\sigma_d} q^{n/N},$
where $\sigma_d$ is the Galois action in $\Gal(\QQ(\zeta_N)/\QQ)$ that sends $\zeta_N \ra \zeta_N^d$.

\end{itemize}

As the restriction maps between the fields $\F_N$ are in correspondence with the natural maps between the groups $\GL_2(\ZZ/N\ZZ)/\{\pm 1\}$, we can take the projective limit to get the isomorphism:
\[
\Gal(\F/\F_1)\cong \GL_2(\widehat{\ZZ})/\{\pm 1\}.
\]

Note that the maps on $\F_N$ are given by projecting $\GL_2(\widehat{\ZZ})/\{\pm 1\}  \ra \GL_2(\ZZ/N\ZZ)/\{\pm 1\}$. To further get all the automorphisms of $\F$ we need to consider the action of $\GL_2( \AAA_{\QQ, f})$. We get the exact sequence:
\[
1 \ra \{\pm 1\} \ra \GL_2( \AAA_{\QQ, f}) \ra \Aut(\F) \ra 1.
\]

For this to make sense, we need to extend the action from $\GL_2(\widehat{\ZZ})$ to $\GL_2( \AAA_{\QQ, f})$. We do this by defining the action of $\GL_2(\QQ)^{+}$ on $\F$:  

\begin{itemize}

 \item \textbf{Action of $\alpha \in \GL_2(\QQ)^{+}$ on $\F$.} We define $f^{\alpha}(\tau)=f(\alpha \tau),$ where $\alpha$ acts by fractional linear transformations.
 
\end{itemize}

We extend the action of $\GL_2(\widehat{\ZZ})$ to $\GL_2(\AAA_{\QQ})$ by writing the elements $g \in \GL_2(\AAA_{\QQ})$ in the form $g=u\alpha$, where $u\in \GL_2(\widehat{\ZZ})$ and $\alpha \in \GL_2(\QQ)^{+}$. Note that this decomposition is not uniquely determined. However, by combining the two actions of $u$ and $\alpha$, a well defined action is given by:
\[
f^{u\alpha}=(f^{u})^{\alpha}.
\]
We want to look at the action of $\Gal(K^{ab}/K)$ inside $\Aut(\F)$. From class field theory we have the exact sequence:
\[
1 \ra K^{\times} \ra  \AAA_{K, f}^{\times} \xrightarrow{[\cdot, K]}  \Gal(K^{ab}/K) \ra 1,
\]
where $[\cdot, K]$ is the Artin map. 

We are going to embed $ \AAA_{K, f}^{\times}$ into $\GL_2( \AAA_{\QQ, f})$ such that the Galois action of $ \AAA_{K, f}^{\times}$ through the Artin map and the action of the matrices in $\GL_2( \AAA_{\QQ, f})$ are compatible. We do this by constructing a matrix $g_{\tau}(x)$ for the idele $x\in  \AAA_{K, f}^{\times}$.

Let $\OO$ be the order of $K$ generated by $\tau$ i.e. $\OO=\ZZ[\tau]$. We define the matrix $g_\tau(x)$ to be the unique matrix in $\GL_2(\AAA_{\QQ})$ such that $x\left(  \begin{matrix} \tau \\ 1\end{matrix}\right)=g_{\tau}(x)\left( \begin{matrix} \tau \\ 1 \end{matrix}\right)$. We can compute it explicitly. To do that, consider the minimal polynomial of $\tau$ to be $p(X)=X^2+BX+C$. Then if we write $x_p \in \QQ_p^{\times}$ in the form  $x_p=s_p\tau+t_p \in \QQ_p^{\times}$ with $s_p, t_p \in \QQ_p$, we can compute $\ds g_{\tau}(x_p)=\left(\begin{matrix}t_p-s_pB & -s_pC \\ s_p & t_p \end{matrix}\right)$.


    %

Using the map $g_{\tau}$ above, we have:

\begin{Thm}\textbf{(Shimura's reciprocity law)} For $f\in \F$ and $x\in \AAA_{K, f}^{\times}$, we have:
\[
(f(\tau))^{\sigma_{x}}=f^{g_{\tau}(x^{-1})}(\tau),
\]
\noindent where $\sigma_x$ is the Galois action corresponding to the idele $x$ via the Artin map, $g_{\tau}$ is defined above and the action of $g_{\tau}(x)$ is the action in $\GL_2( \AAA_{\QQ, f})$.

\end{Thm}

Note that the elements of $K^{\times}$ have trivial action. This can be easily seen by embedding $K^{\times} \hookrightarrow \GL_2(\QQ)^{+}$ via $k\hookrightarrow g_{\tau}(k)$. Noting that $\tau$ is fixed by the action of the torus $K^{\times}$, we have $f^{g_{\tau}(k^{-1})}(\tau)=f(g_{\tau}(k^{-1})\tau)=f(\tau).$

We can also rewrite the theorem for ideals in $K$. Let $f \in \F_N$ and $\OO=\ZZ[\tau]$ of conductor $M$. Going through the Artin map, we can restate Shimura's reciprocity in this case in the form:
\begin{equation}\label{Shimura_rec}
f(\tau)^{\sigma_{\A}}=f^{g_{\tau}(\A)^{-1}}(\tau),
\end{equation} 
\noindent where $\A$ is an ideal prime to $MN$, $\sigma_{\A}$ is the Galois action corresponding to the ideal $\A$ through the Artin map, and $g_{\tau}(\A):=g_{\tau}((\alpha)_{p|\Nm(\A)}).$ Note that $g_{\tau}(\A)$ is unique up to multiplication by roots of unity in $K$. However, these have trivial action on $f$ at the unramified places $p|\Nm(\A)$.

\subsection{Galois conjugates of $f(\omega)$}

We denote $\ds f(z)=\frac{\Theta_K(Dz)}{\Theta_K(z)}$. We are interested in finding the Galois conjugates of $f(\omega)$. First we show that $f(z)$ is a modular function:


\begin{Lem}\label{modular_fourier} The function f(z) is a modular function of level $3D$ with integer Fourier coefficients at the cusp $\infty$.

\end{Lem}

\begin{proof} Since $\Theta_K(z)$ is a modular form of weight $1$ for $\Gamma_1(3)$, it can be easily seen that $\Theta(Dz)$ is a modular form of weight $1$ for $\Gamma(3D)$. Furthermore, their ratio is modular function for $\Gamma_0(3D)$. To check this let $g=\left(\begin{smallmatrix} a& b \\ c & d \end{smallmatrix} \right) \in \Gamma(3D)$, and then it is easy to see that $f(gz)=\frac{\Theta\left(\left(\begin{smallmatrix} D& 0 \\ 0 & 1 \end{smallmatrix} \right) gz\right)}{\Theta\left(gz\right)}
=
\frac{\Theta\left(\left(\begin{smallmatrix} a& bD \\ c/D & d \end{smallmatrix} \right)(Dz)\right)}{\Theta\left(\left(\begin{smallmatrix} a& b \\ c & d \end{smallmatrix} \right)z\right)}
=
f(z).$

To find the Fourier expansion of $f(z)$ at $\infty$, it is enough to write the Fourier expansions of $\Theta(Dz)$ and $\Theta(z)$: $\frac{\Theta(Dz)}{\Theta(z)}
=
\frac{1+\sum\limits_{N\geq 1}c_{N}q^{ND}}{1+\sum\limits_{N\geq 1}{c_{N}}q^N}=\sum\limits_{M\geq 0}a_M q^M$. We can compute the Fourier coefficients explicitly from the equality: $a_0=1$ and $a_M=-a_{M-1}c_{1}-a_{M-2}c_{2}-\dots-a_1c_{M-1}-a_0c_{M}$ if $D\nmid M$ and $a_M=c_{M/D}-a_{M-1}c_{1}-a_{M-2}c_{2}-\dots-a_1c_{M-1}-a_0c_{M}$ if $D|M$. By induction, since $c_N\in \ZZ$, we get all the coefficients $a_M \in \ZZ$.\end{proof}

 From CM-theory, if $f\in \F_{3D}$ and $\tau$ a generator of $\OO_K$, we have $f(\tau)\in H_{\OO}$ the ray class field of modulus $3D$. Recall $H_{3D}$ is the ring class field for the the order $\OO_{3D}=\ZZ+3D\OO_K$, and we actually have: 

\begin{Lem}\label{lemma_H_3D} $f(\omega)\in H_{3D}$.

\end{Lem}

\begin{proof} We need to show that $f(\omega)$ is invariant under $\Gal(K^{ab}/H_{3D})$. Recall that we have $\Gal(H_{3D}/K)\cong U(3D)\setminus \AAA_{K, f}^{\times}/K^{\times}$ for $U(3D)=(1+3\ZZ_3[\omega])\prod\limits_{v|D}{(\ZZ+D\ZZ_p[\omega])^{\times}}\prod\limits_{v\nmid 3D}{\OO_{K_v}^{\times}}$. We check that $f(\omega)$ is invariant under the action of $K^{\times}U(3D)$. Using Shimura's reciprocity law, we want to show:
\[
f(\omega)=f^{g_{\omega}(s)}(\omega), 
\]
\noindent for all $s\in K^{\times}U(3D)$. We noted before that the action of $K^{\times}$ is trivial. Thus it is enough to show the result for all elements $l=(A_p+B_p\omega)_{p}\in U(3D)$. By the definition of $U(3D)$, this implies that $A_p+B_p\omega \in (\ZZ_p[\omega])^{\times}$ for all $p$  and $A_3 \equiv 1\mod 3$, $B_3 \equiv 1\mod 3$, $B_p \equiv 0 \mod D$ for all $p|D$. Since the action for $p\nmid 3D$ is trivial, $l$ has the same Galois action as $l_D=(A_p+B_p\omega)_{p|3D}\in U(3D)$. Moreover, this has the same action as $l_0=(A+B\omega)_{p|3D}$, where $A+B\omega \in \OO_K$ and $A\equiv A_p \mod 3D\ZZ_p$ and $B\equiv B_p\mod 3D\ZZ_p$ for all $p|3D$. 

Note further that we can pick $A, B$ such that $(A+B\omega)$ generates a primitive ideal $\A$  in $\OO_K$. Moreover, from above we have $3D|B$ and $A\equiv 1\mod 3$. Recall that we can rewrite any primitive ideal in the form $\A=[a, \frac{-b+\sqrt{-3}}{2}]_{\ZZ}$, where $a=\Nm\A$ and $b^2\equiv -3 \mod 4a$. Then the generator of the ideal $\A$ is $A+B\omega=ta+s\frac{-b+\sqrt{-3}}{2}$ for $t, s\in \ZZ$ and $3D|s$.

Now note that $f(\omega)=f(\tau)$, where $\tau=\frac{-b+\sqrt{-3}}{2}$. Thus from Shimura's reciprocity law, we have:
\[
(f(\tau))^{\sigma_{l^{-1}}}=f^{g_{\tau}(l_{p|3D})}(\tau)=f^{g_{\tau}(l_{0})}(\tau).
\]
Here $g_{\tau}(l_0)=\left(\begin{smallmatrix} ta-sb & -sca \\ s & ta \end{smallmatrix} \right)_{p|3D}$, where $ca=\frac{b^2+3}{4}$. Then we can rewrite the action of $g_{\tau}(l_0)$ explicitly as:
\[
f^{g_{\tau}(l_0)}(\tau)
=
f^{\left(\begin{smallmatrix} ta-sb & -sc \\ s & t \end{smallmatrix} \right)_{p|3D}\left(\begin{smallmatrix} 1 & 0 \\ 0 & a \end{smallmatrix} \right)_{p|3D}}(\tau)
=
f^{\left(\begin{smallmatrix} 1 & 0 \\ 0 & a \end{smallmatrix} \right)_{p|3D}}(\left(\begin{smallmatrix} ta-sb & -sc \\ s & t \end{smallmatrix} \right)\tau).
\]
Since $3D|s$, the matrix $\left(\begin{smallmatrix} ta-sb & -sc \\ s & t \end{smallmatrix} \right)\in \Gamma_0(3D)$ and $f(z)$ is invariant under its action. Finally, since $(a, 3D)=1$ and $f$ has rational Fourier coefficients at $\infty$, the action of $\left(\begin{smallmatrix} 1 & 0 \\ 0 & a \end{smallmatrix} \right)_{p|3D}$ is trivial. Thus $f(\omega)$ is invariant under the Galois action coming from $U(3D)$ and this finishes the proof. \end{proof}

Now we would like to compute the Galois conjugates of $f(\omega)$ under the action of $\Gal(H_{3D}/K)$. We will first show the following general result:

\begin{Lem}\label{galois_conj} Let $F\in \F_N$ be a modular function of level $N$ with rational Fourier coefficients in its Fourier expansion at $\infty$. Let $\tau=\frac{-b+\sqrt{-3}}{2}$ be a CM point and let $\A=\left[a, \frac{-b+\sqrt{-3}}{2}\right]_{\ZZ}$ be a primitive ideal prime to $N$. Then we have the Galois action:
\[
F(\tau)^{\sigma_{\A}^{-1}}=F(\tau/a).
\]
\end{Lem}

\begin{proof} From Shimura's reciprocity law, we have $F(\tau)^{\sigma_{\A}^{-1}}=F^{g_{\tau}(\A)} (\tau).$ Note that the minimal polynomial of $\tau$ is $p_{\tau}(X)=X^2+bX+\frac{b^2+3}{4}$. Let $\alpha$ be a generator of $\A$. Then we can write $\alpha$ in the form $\alpha=ta+s\tau$ and we have $g_{\tau}(\A)=\left(\begin{smallmatrix} ta-sb & -s\frac{b^2+3}{4} \\ -s & ta\end{smallmatrix}\right)_{p|a}$. We can rewrite the matrix in the form $g_{\tau}(\A)
=
\left(\begin{smallmatrix} ta-sb & \frac{b^2+3}{4a} \\ -s & t\end{smallmatrix}\right)_{p|a}
\left(\begin{smallmatrix} 1 & 0\\ 0 & a \end{smallmatrix}\right)_{p|a}.$ As $\left(\begin{smallmatrix} ta-sb & -\frac{b^2+3}{4a} \\ -s & t\end{smallmatrix}\right)_{p|a}$ is an element of $\SL_2(\ZZ_p)$ for $p\nmid N$, it has trivial action. Then we have $F^{g_{\tau}(\A)}(\tau)
=
F^{\left(\begin{smallmatrix} 1 & 0\\ 0 & a \end{smallmatrix}\right)_{p|a}}(\tau).$ 

We rewrite the matrix $\left(\begin{smallmatrix} 1 & 0\\ 0 & a \end{smallmatrix}\right)_{p|a}=\left(\begin{smallmatrix} 1 & 0\\ 0 & 1/a \end{smallmatrix}\right)_{p\nmid a}\left(\begin{smallmatrix} 1 & 0\\ 0 & a \end{smallmatrix}\right)_{\QQ}$, where $\left(\begin{smallmatrix} 1 & 0\\ 0 & 1/a \end{smallmatrix}\right)_{p\nmid a}\in \GL_2(\widehat{\ZZ})$ and $\left(\begin{smallmatrix} 1 & 0\\ 0 & a \end{smallmatrix}\right)_{\QQ}\in \GL_2(\QQ)^{+}$. Note that the action of $\left(\begin{smallmatrix} 1 & 0\\ 0 & 1/a \end{smallmatrix}\right)_{p\nmid a}$ is the same as the action of $\left(\begin{smallmatrix} 1 & 0\\ 0 & 1/a \end{smallmatrix}\right)_{p|N}$. However, since $F$ has rational Fourier coefficients in its Fourier expansion, this action is trivial. Thus we are left with $F^{g_{\tau}(\A)}(\tau)
=
F^{\left(\begin{smallmatrix} 1 & 0\\ 0 & a \end{smallmatrix}\right)_{\QQ}}(\tau)
=
F(\tau/a)$, which finishes the proof.

\end{proof}


We apply the lemma above to our case:

\begin{Prop}\label{Galois_conj_theta} Take the primitive ideals $\A=\left[a, \frac{-b+\sqrt{-3}}{2}\right]_{\ZZ}$ to be the representatives of the classes of the ring class group $\Cl(\OO_{3D})$ such that all norms $a=\Nm\A$ are relatively prime to each other and $b^2\equiv -3 \mod 4a$ for all the norms $a$.

Then the Galois conjugates of $f(\omega)=\frac{\Theta_K(D\omega)}{\Theta_K(\omega)}$ are the terms $\left(\frac{\Theta_K(D\omega)}{\Theta_K(\omega)}\right)^{\sigma_{\A}^{-1}}
=
\frac{\Theta_K\left(D\tau_{\A}\right)}{\Theta_K\left(\tau_{\A}\right)}.$

\end{Prop}

\begin{proof} We note that $\frac{\Theta_K(D\omega)}{\Theta_K(\omega)}=\frac{\Theta_K\left(D\tau\right)}{\Theta_K\left(\tau\right)}$ and apply Lemma \ref{galois_conj} to $\tau=\frac{-b+\sqrt{-3}}{2}$ and $f(z)=\frac{\Theta_K(Dz)}{\Theta_K(z)}$. These are the only Galois conjugates, as we showed that $f(\tau)\in H_{3D}$. \end{proof}

\subsection{Galois conjugates of $f_{r, \mu}(\tau)$}\label{shimura_r}

Recall that we defined $\ds f_{r, \mu}(z)=\frac{\theta_{r, \mu}(z)}{\theta_{0}(z)}$. This is a modular function for $\Gamma(18D^2)$ when $\mu=1/2$ and for $\Gamma(9D^2)$ when $\mu=1/6$. From CM-theory then $f_{r, \mu}(\tau)\in H_{\OO'}$, where $H_{\OO'}$ is the ray class field of modulus $18D^2$ for $\mu=1/2$ and of modulus $9D^2$ for $\mu=1/6$. In order to compute the Galois conjugates of $f_{r, \mu}(\tau)$ over $K$ we can use Shimura's reciprocity law. 


In our case we want to compute the Galois conjugates of $f_{r, \mu}(\tau)$ for $\tau=\frac{-b+\sqrt{-3}}{2}$, with $b^2\equiv -3\mod 12D^2$. Note that $\tau$ has the minimum polynomial $X^2+bX+\frac{b^2+3}{4}$. Thus we have to compute the action of all $\ds g_{\tau}((x_p)_p)=\prod\limits_p\left( \begin{smallmatrix} t_p-s_pb & -s_p\frac{b^2+3}{4} \\ s_p & t_p \end{smallmatrix} \right)_p$ on $f_{r, \mu}(\tau)$.  We will prove that the Galois action from Shimura's reciprocity law is given by the following:

\begin{Prop}\label{Galois_theta_r} 
 For $\A=[a, \frac{-b+\sqrt{-3}}{2}]_{\ZZ}$ an ideal prime to $6D$ such that $b^2\equiv -3 \mod 12Da^2$. Let $k_{\A}=n_aa+m_a\frac{-b+\sqrt{-3}}{2}$ be the generator of $\A$ such that $3|m_a$ and $n_a\equiv 1 \mod 3$, then for $\tau=\frac{-b+\sqrt{-3}}{2}$ we have:
\[
f_{r, \mu}(\tau)^{\sigma_{\A}^{-1}}=f_{n'_{a}r, \mu}(\tau),
\]
where $n_a\equiv n'_a \mod 3D$ and $n_{a'}$ odd.
Moreover, these are all the Galois conjugates of $f_{r, \mu}(\tau)$ and $f_{r, \mu}(\tau)$ is in $H_{\OO}$ the ray class field of modulus $3D$.
\end{Prop}
	
{\it Proof.} We will compute the Galois conjugates of $f_{r, \mu}(\tau)$ using Shimura's reciprocity law. As we noted before, it is enough to look at the action of $K^{\times} \prod\limits_{v\nmid \infty}\OO_{K_v}^{\times}$. Moreover, we also noted that the action of $K^{\times}$ is trivial.

To compute the action of $\prod_v \OO_{K_v}^{\times}$ note first that for all $v\nmid 6D$ the action is trivial. For $v|6D$ we project the action of  $(g_{\tau}(x_v))_v$  to $g_{\tau}(x')\in \GL_2(\ZZ/18D^2\ZZ)$.  From the Chinese remainder theorem, we can find $k_0\in K$ such that $k_0 \equiv x_p\mod 72D^2\ZZ_p$ for all $p|6D$. Note that $k_0$ is independent of the choice of $\tau$ and $g_{\tau}(x) \equiv  g_{\tau}(k_0)$ in $\GL_2(\ZZ/18D^2\ZZ)$.

	Moreover, $(\pm\omega^i)_p \hookrightarrow \AAA_{K_f}^{\times}$ acts trivially, thus we can consider the action of $k_0(\pm \omega^j)$ instead, for $j \in \{0, 1, 2\}$. We pick $k_0=A+B\omega$ such that $v_3(B)\geq 1$. Moreover, by eventually changing $b$ to $b+18D^2$, we can pick $b'\equiv b \mod 18D^2\Nm k_0$ such that $k_0=ta+s\frac{-b'+\sqrt{-3}}{2}$ with $ta-sb' \equiv 1\mod 6$. Let $\tau'=\frac{}{-b'+\sqrt{-3}}{2}$ and we have:
\[
(f_{r, \mu}(\tau))^{\sigma_{(k_0)_{6D}}}
=
(f_{r, \mu}(\tau'))^{\sigma_{(k_0)_{6D}}}
=
f_{r, \mu}^{g_{\tau'}((k_0)_{6D})}(\tau').
\]

We have $g_{\tau'}((k_0)_{6D})=\left(\begin{smallmatrix} ta-sb' & -sc'a \\ s & ta \end{smallmatrix}\right),$ where $c'a=\frac{b'^2+3}{4}$.

We will now compute $f_{r, \mu}^{g_{\tau}(x)_{p|6D}}(\tau)$. We write the matrix $g_{\tau'}((k_0)_{6D})$ as a product $g_{\tau'}((k_0)_{6D})=\left(\begin{smallmatrix} ta-sb & -sc' \\ s & t \end{smallmatrix}\right)_{p|6D}\left(\begin{smallmatrix} 1 & 0 \\ 0 & a \end{smallmatrix}\right)_{p|6D}$. Note that $\left(\begin{smallmatrix} 1 & 0 \\ 0 & a \end{smallmatrix}\right)_{p|6D}$ acts trivially on $f_{r, \mu}$ as the functions $\theta_{r, \mu}(z)$ and $\theta_0(z)$ have rational Fourier coefficients in its Fourier expansion at $\infty$, and thus so does $f_{r, \mu}(z)$. Thus we need to compute the action $\ds f_{r, \mu}^{\left(\begin{smallmatrix} ta-sb & -sc' \\ s & t \end{smallmatrix}\right)_{p|6D}}(\tau')$.  Note that this is a matrix in $\SL_2(\ZZ)$ and it acts as:
\[
 f_{r, \mu}^{\left(\begin{smallmatrix} ta-sb' & -sc' \\ s & t \end{smallmatrix}\right)_{p|6D}}(\tau')
 =
 f_{r, \mu}(\left(\begin{smallmatrix} ta-sb' & -sc' \\ s & t \end{smallmatrix}\right)\tau').
\]
We can further compute this transformation and we will do this explicitly in Lemma \ref{transformation_r}. As we have $3D^2|s$ and $ta-sb\equiv 1\mod 6$ we are in the conditions of this lemma. Applying the transformation for $\theta_{r, \mu}$ and $\theta_0$ and moreover noting that $9|sc'$ we get precisely:
\[
f_{r, \mu}(\left(\begin{smallmatrix} ta-sb' & -sc' \\ s & t \end{smallmatrix}\right)\tau')=\frac{\theta_r\left(\left(\begin{smallmatrix} ta-sb' & -sc' \\ s & t \end{smallmatrix}\right)\tau'\right)}{\theta_0\left(\left(\begin{smallmatrix} ta-sb' & -sc' \\ s & t \end{smallmatrix}\right)\tau'\right)} 
=
f_{(ta-sb')r, \mu}(\tau').
\]

Since $(ta-sb')t \equiv 1\mod D^2$, we can rewrite this as $f_{t'^{-1}r, \mu}(\tau')$ for $t\equiv t \mod D$ and $t'\equiv 1 \mod 6$. Note that $t$ is prime to $D$. Thus we have showed so far that the Galois conjugates of $f_{r, \mu}(\tau)$ are the terms $f_{s, \mu}(\tau)$, where $\gcd(s, D)=\gcd(r, D)$. Moreover, we have nontrivial Galois action only for $k_0=ta+s\frac{-b+\sqrt{-3}}{2}$ with $t\not\equiv 1\mod D$. Furthermore, it implies that $f(\tau)\in H_{\OO}$, the ray class field of modulus $3D$.

Finally, we would like to express the Galois action using ideals. For $\A=[a, \frac{-b+\sqrt{-3}}{2}]_{\ZZ}$ a primitive ideal prime to $6D$ with a generator $(k_{\A})=(n_aa+m_a\frac{-b+\sqrt{-3}}{2})$ with $n_a\equiv 1 \mod 6$ and $3|m_a$, we have the correspondence map between ideles and ideals given by $\ds x=(k_{\A})_{p\nmid 6D} \leftrightarrow \A=(k_{\A})$. Picking the representatives $k_{\A}$ as above, we have:
\[
f_{r, \mu}(\tau)^{\sigma_{\A}}
=
f_{r, \mu}^{g_{\tau}^{-1}(k_{\A})_{p\nmid 6D}}(\tau)
=
f_{r, \mu}^{g_{\tau}(k_{\A})_{p|6D}}(\tau)
=
f_{n_{a}'^{-1}r, \mu}(\tau),
\]
where $n_a\equiv n'_a \mod 3D$ and $n'_a$ odd. After changing $r \ra n_{a}'^{-1} r$, we get the result of the Galois action from the proposition.





\section{Appendix: Properties of theta functions}

\subsection{Properties of $\Theta_K$ and $\eta$}
We have a functional equation for the theta function (see \cite{Koh}):

\begin{equation}\label{functional_eq}
\Theta_K(-1/3z)=\frac{3}{\sqrt{-3}}z\Theta_K(z).
\end{equation}

Furthermore, we can compute the transformation of $\Theta_K(z\pm1/3)$ in the lemma below:

\begin{Lem}\label{theta3z} $\ds \Theta\left(z+k/3\right)=(1-\omega^{2k})\Theta(3z)+\omega^{2k}\Theta(z)$ for $k\in \ZZ$.

\end{Lem}

\begin{proof}

 For $k=1$, we can split the sum $\Theta\left(z+\frac{1}{3}\right)=\sum\limits_{m, n \in \ZZ}e^{2\pi i (m^2+n^2-mn)\left(z+\frac{1}{3}\right)} $ in two parts, depending on whether or not the ideal $(m+n\omega)$ is prime to $(\sqrt{-3})$. The part of the sum for which $(\sqrt{-3})|(m+n\omega)$ gives us $\sum\limits_{m, n \in \ZZ}e^{2\pi i (m^2+n^2-mn)\left(3z+1\right)}
=
\Theta(3z+1)
=
\Theta(3z).$ 

The part of the sum for which $(\sqrt{-3})\nmid (m+n\omega)$ can be rewritten as $\omega\sum\limits_{\substack{m, n \in \ZZ, \\ (\sqrt{-3})\nmid (m+n\omega)}}e^{2\pi i (m^2+n^2-mn)z}$ as $m^2+n^2-mn \equiv 1(3)$. We rewrite this sum as the sum of two terms $\omega\sum\limits_{m, n \in \ZZ}e^{2\pi i (m^2+n^2-mn)z}-\omega\sum\limits_{\substack{m, n \in \ZZ, \\ (\sqrt{-3})|(m+n\omega)}}e^{2\pi i (m^2+n^2-mn)z}$. Finally we recognize the two terms as $\omega\Theta(z)-\omega\Theta(3z).$ 

Going back to our initial computation, we get $\Theta\left(z+1/3\right)=\Theta(3z)+\omega\Theta(z)-\omega\Theta(3z)=(1-\omega)\Theta(3z)+\omega\Theta(z),$ and this finishes the proof of the first formula. We can show the case $k=2$ by applying the equality for $k=1$ and $z:=z-1/3$. \end{proof}

By applying the functional equation (\ref{functional_eq}) for $z=\frac{-3+\sqrt{-3}}{6}$ we get the following easy lemma:

\begin{Lem} $\Theta_K\left(\frac{-3+\sqrt{-3}}{6}\right)=0$.
\end{Lem}

We have further:

\begin{Lem}\label{theta_03} For a primitive ideal $\A=[a, \frac{-b+\sqrt{-3}}{2}]_{\ZZ}$ prime to $3$ such that $b^2\equiv -3 \mod 12a$, we have $\Theta_K\left(\frac{-b+\sqrt{-3}}{6a}\right)=0.$
\end{Lem}

\begin{proof} The proof is similar to that of Lemma \ref{theta_omega}. We can write the generator of primitive ideal $\A=\left[a, \frac{-b+\sqrt{-3}}{2}\right]$ in the form $\ds k_{\A}=ma+n\frac{-b+\sqrt{-3}}{2}$ for some integers $m, n$ and following similar steps as in Lemma \ref{theta_omega}, we get $\Theta_K\left(\frac{-b+\sqrt{-3}}{6}\right)=(m-n\frac{-b+\sqrt{-3}}{2a})\Theta_K\left(\frac{-b+\sqrt{-3}}{6a}\right).$ From the previous lemma, we know the LHS equals $0$, thus $\Theta_K\left(\frac{-b+\sqrt{-3}}{6a}\right)=0$ as well. \end{proof}

\bigskip
We have also used the following lemma in the proof of Proposition \ref{square2}:

\begin{Lem}\label{eta_lem} For $b^2\equiv -3 \mod 12D^2$ and $\pi$ of norm $D$ such that the ideal $(\pi)^2$ divides $\left(\frac{-b+\sqrt{-3}}{2}\right)$, we have for $\tau=\frac{-b+\sqrt{-3}}{2}$:
\[
\frac{\eta\left(\tau/D^2\right)}{\eta\left(\tau\right)}=(-1)^{\frac{D-1}{6}}\overline{\pi}
\]
\end{Lem}
{\it Proof.} We apply Corollary \ref{theta_0} twice to get: $\ds (-1)^{\frac{D-1}{6}}\frac{\theta_0(\tau/D^2)}{\theta_0(\tau)}
\frac{\theta_0\left(\overline{\tau}\right)}{\theta_0\left(\overline{\tau}\right)}
=
\frac{\Theta_K\left(\tau/D\right)}{\Theta_K\left(\tau\right)}.$

Then from Lemma \ref{theta_omega}, we have the RHS equal to $\overline{\pi}$. Furthermore, we can pick $b\equiv b'\mod 8D^2$, $b\equiv 0\mod 3$ and $b'\equiv 1\mod 3$. Denote $\tau'=\frac{-b'+\sqrt{-3}}{2}$.  We can pick without loss of generality $b, b'$ such that $(b^2+3)/D^2$ and $(b'^2+3)/D^2$ are prime to $D$. Then we can find ideals $\A, \A'$ prime to $D$ such that $\A(\pi)^2(\sqrt{-3})=(\tau)$ and $\A'(\pi)^2=(\tau')$. Let $a=\Nm\A, a'=\Nm\A'$ and then we have:
\[
\frac{\theta_0\left(\tau/D^2\right)}{\theta_0\left(\tau\right)}
=
\left(\frac{\theta_0\left(\tau/D^2\right)}{\theta_0\left(\tau\right)}\right)^{\sigma_{\A}^{-1}}
=
\frac{\theta_0\left(\frac{\tau}{aD^2}\right)}{\theta_0\left(\tau/a\right)}
=
\frac{\eta\left(\frac{\tau}{3aD^2}\right)}{\eta\left(\frac{\tau}{3a}\right)}
=
\frac{\eta\left(\overline{\tau}\right)}{\eta\left(\overline{\tau}/D^2\right)}.
\]

Similarly we compute $\ds \frac{\eta(\overline{\tau})}{\eta\left(\overline{\tau}/D^2\right)}
=
\frac{\eta(\overline{\tau'})}{\eta\left(\overline{\tau'}/D^2\right)}
=
\frac{\eta\left(\frac{\tau'}{a'D^2}\right)}{\eta\left(\frac{\tau'}{a'}\right)}
=
\left(\frac{\eta\left(\tau'/D^2\right)}{\eta\left(\tau'\right)}\right)^{\sigma_{\A'}^{-1}}.$

Note that we also have $\frac{\eta\left(\frac{\tau'}{D^2}\right)}{\eta\left(\tau'\right)}=\frac{\eta\left(\frac{\tau}{D^2}\right)}{\eta\left(\tau\right)}$, and thus we have $\ds \frac{\eta\left(\frac{\tau}{D^2}\right)}{\eta\left(\tau\right)}
=
(-1)^{\frac{D-1}{6}}\overline{\pi}.$

\subsection{Properties of $\theta_{r, \mu}$}

Recall that for $r\in \ZZ$, $\mu \in \{1/2, 1/6\}$, we have defined the theta function $\ds \theta_{r, \mu}(z)=\sum\limits_{n\in \ZZ}e^{\pi i (n+r/D-\mu)^2z} (-1)^n.$

We will write $\theta_{r, \mu}$ as an automorphic theta function. Using the standard notation:
 \[
m(a)=\left(\begin{matrix} a & 0  \\ 0 & a^{-1} \end{matrix}\right),  n(b)=\left(\begin{matrix} 1 & b \\0 & 1 \end{matrix}\right), w=\left(\begin{matrix} 0 & 1 \\ -1 & 0 \end{matrix}\right),
\]
for $\phi\in\SSS(\AAA_{\QQ})$ a Schwartz-Bruhat function the Weil representation $r$ for $\SL_2(\AAA_{\QQ})$ is defined by:

\begin{itemize}

\item $r\left(m(a)\right)\phi(x)=\chi_{0}(a)|a|^{1/2} \phi(ax)$

\item $r\left(n(b)\right)\phi(x)
=
\psi(bx^2) \phi(x)$

\item $r\left(w\right)\phi(x)=\gamma \widehat{\phi}(x),$

\end{itemize}

\noindent where $\psi_p(x)=e^{-2\pi i \Frac_p(x)}$ and $\psi_{\infty}(x)=e^{2\pi i x}$, $\gamma$ is an 8th root of unity, and $\chi_0$ is a quadratic character. We chose the self-dual Haar measure such that $\widehat{\widehat{\phi}}(-x)=\phi(x)$.

We define the Schwartz-Bruhat functions $\phi^{r, \mu}=\prod\limits_v \phi_v^{r, \mu}$ for $\theta_{r, \mu}$ by taking $\phi^{r, \mu}_{\infty}(x)=e^{-2\pi x^2 }$, $\phi^{r, \mu}_p = \Char_{\ZZ_p+\frac{r}{D}-\mu}$ for $p\neq 2$ and $\phi^{r, \mu}_2(x) = e^{\pi i \Frac_2(x) }\Char_{\ZZ_2-1/2}(x)$. Then for the theta function 
\[
\theta(g, \phi^{r, \mu})=\sum_{x\in \QQ} r(g)\phi^{r, \mu}(x),
\] 
for $g_z=\left(\begin{smallmatrix} y^{1/2} & y^{-1/2}x \\ 0 & y^{-1/2} \end{smallmatrix}\right)$, we can compute $ \theta(g_z, \phi^{r, \mu})
=
e^{-\pi i \Frac_2 \mu}(-1)^ry^{1/4}\theta_{r, \mu}(2z).$

Using the properties of the Weil representation and the definition above for $\theta_{r, \mu}$, after a straightforward computation we get the following lemma:

\begin{Lem}\label{transformation_r} For $\left(\begin{smallmatrix} a& b \\ c & d \end{smallmatrix}\right)\in \SL_2(\ZZ)$ such that $3D^2|b$ and $a\equiv 1 \mod 6$, we have the transformation:
\[
\theta_{r, \mu}\left(\frac{az+b}{cz+d}\right)
=
\sgn(d)e^{\pi i (a-1)/2} \chi_{0, 6}(a)e^{2\pi i \Frac_2(\frac{ba-c/a}{8})}e^{2\pi i \Frac_3(t_{\mu}\frac{ba/2}{9})} \sqrt{cz+d}
\theta_{ar, \mu}(z).
	\]
\noindent where $t_{1/2}=0$ and $t_{1/6}=1$.
	
\end{Lem}

\begin{proof} Let $\phi=\phi^{\mu, r}$. Recall that $ \theta(g_{z/2}, \phi)
= e^{-\pi i \Frac_2 \mu}(-1)^r(y/2)^{1/2}\theta_{r, \mu}(z).$. Note that  $m(\sqrt{2})^{-1} \left(\begin{smallmatrix} a& b \\ c & d \end{smallmatrix}\right)
=
\left(\begin{smallmatrix}a & b/2 \\ 2c & d \end{smallmatrix}\right)
m(\sqrt{2})^{-1}$. As $\theta(g, \phi)$ is invariant under the action of $\SL_2(\QQ)$, we have 
\begin{equation}\label{eq_0}
\theta(m(\sqrt{2})^{-1}\left(\begin{smallmatrix}a & b \\ c & d \end{smallmatrix}\right)
g_{z}, 1_f, \phi)
=
\theta (m(\sqrt{2})^{-1}g_{z}, \left(\begin{smallmatrix}a & b/2 \\ 2c & d \end{smallmatrix}\right)_f^{-1}, \phi ).
\end{equation}

We will compute separately the LHS and the RHS using the definition of the Weil representation. We compute first the RHS. In order to do this, we rewrite the matrix $\left(\begin{smallmatrix}d & -b/2 \\ -2c & a \end{smallmatrix}\right)
=
m(a^{-1})n(-ba/2) m(-1)w\cdot n(2c/a)\cdot w$.

At $p\nmid 6D$, the action of $\left(\begin{smallmatrix}d & -b/2 \\ -2c & a \end{smallmatrix}\right)
$ is trivial. For $p|D$, we can easily compute the Fourier transforms $\widehat{\phi_p}(x)= e^{-2\pi i \Frac_p(2rx/D)}\Char_{\ZZ_p}(x)$ using $v_p(c/a)\geq 0$ we get $r(m(-1)w\cdot n(2c/a)\cdot w)\phi_p=\phi_p$. Furthermore using $v_p(ba/2)\geq 2$, we get $r(m(a^{-1})n(-ba/2))\phi_p(x)=\phi_p(x/a)=\Char_{\ZZ_p+\frac{ar}{D}}(x)$.

For $p=3$, the computation is similar. For $\mu=1/2$ we get $\chi_{0, 3}(a)\phi_3(x)$, while for $\mu=1/6$ we get $\chi_{0, 3}(a)e^{-2\pi i \Frac_3(\frac{ba/2}{9})}\phi_3(x)$.


For $p=2$, we have  the Fourier transform $\widehat{\phi_2}(x)= e^{\pi i/2}e^{-2\pi i\Frac_2(x)}\Char_{\frac{1}{2}(\ZZ_2+1/2)}(x)$. Using $v_2(2c/a)\geq 1$ we get $r(n(2c/a))\widehat{\phi_2}(x)=e^{-2\pi i \Frac_2(\frac{c/a}{8})}\widehat{\phi_2}(x)$ and as $v_2(ba)\geq 0$ we have $r(n(-ba/2))\phi_2(x)=e^{2\pi i \Frac_2(\frac{ba}{8})}\phi_2(x)$. Thus we get $\chi_{0, 2}(a)e^{2\pi i \Frac_2(\frac{ba-c/a}{8})}\phi_2(x/a)$ which equals $\chi_{0, 2}(a)e^{\pi i (a-1)/2} e^{2\pi i \Frac_2(\frac{ba-c/a}{8})}\phi_2(x)$.

This finishes the computation as we got  $\theta(m(\sqrt{2})^{-1}g_{z}, \left(\begin{smallmatrix}a & b/2 \\ 2c & d \end{smallmatrix}\right)^{-1}_f, \phi)$ to equal:
\begin{equation}\label{finite_comp}
c(-1)^re^{-\pi i \Frac_2 \mu}(y/2)^{1/4}\theta_{ar, \mu}(z),
\end{equation}
where $c=\chi_{0, 6}(a) e^{\pi i (a-1)/2} e^{2\pi i \Frac_2(\frac{ba-c/a}{8})}e^{2\pi i t_{\mu} \Frac_3(\frac{ba/2}{9})}$.

We will compute now the LHS of (\ref{eq_0}). We have $r(g_z)\phi_{\infty}(m)=y^{1/4} e^{2\pi i zm^2}$. We rewrite the matrix $\left(\begin{smallmatrix}
a & b \\ c & d
\end{smallmatrix}\right)
=
 n(b/d) m(d^{-1})m(-1) w n(-c/d) w.$ We compute the Weil representation action and get 
 \[
r(m(\sqrt{2})^{-1})r\left(\begin{smallmatrix}
a & b \\ c & d
\end{smallmatrix}\right)r(g_z)\phi_{\infty}(x) =(y/2)^{1/4}\sgn(d)\sqrt{\frac{1}{cz+d}}e^{\pi i\left(\frac{az+b}{cz+d}\right)x^2},
 \]
and thus we have:
\begin{equation}\label{infinite_comp}
\theta\left(m(\sqrt{2})^{-1}\left(\begin{smallmatrix} a& b \\ c & d \end{smallmatrix}\right) g_{z}, 1_f, \phi \right)
=
(-1)^re^{-\pi i \Frac_2\mu}(y/2)^{1/4}\sgn(d)\sqrt{\frac{1}{cz+d}}	\theta_r\left(\frac{az+b}{cz+d}\right)
\end{equation}

From (\ref{finite_comp}) and (\ref{infinite_comp}) we get the result of the lemma.
\end{proof}

It follows immediately by applying the lemma above for $\theta_{r, \mu}$ and $\theta_0$ that:

\begin{Lem}
$f_{r, 1/2}$ is a modular function for $\Gamma(18D^2)$ and $f_{r, 1/6}$ is a modular function for $\Gamma(6D^2)$. 
\end{Lem}

\begin{rmk} Also from Lemma \ref{transformation_r} it is easy to see that $f_{r, \mu}(z+9D^2)=f_{r, \mu}(z)$.
\end{rmk}

We can also compute the transformation under $w=\left(\begin{smallmatrix} 0 & 1 \\ -1 & 0 \end{smallmatrix}\right)$ of $\theta_{r, \mu}$. This is also done by straightforward computation. We get:
	
\begin{Lem}\label{Fourier}  We have the transformation:
\[
\ds \theta_{r, 1/6}(3z)=(-1)^r\frac{\omega e^{\pi i \frac{D-1}{6}}}{\sqrt{-3}\sqrt{-iz}}\left(\theta^{(-3r), 1/6}(3(-1/z))-\omega \theta^{(3r), 1/6}(3(-1/z))-\omega^2\theta^{(-3r), 1/2}(3(-1/z))\right)
\]

\noindent where $\ds \theta^{(r), \mu}(z)=\sum\limits_{n\in \ZZ} e^{\pi i (n-\mu)^2 z} (-1)^n e^{2\pi i nr/D}$.

\end{Lem}
{\it Proof.} Denote $\phi=\phi^{r, 1/6}$. Then for $g_{Z/2}=\left(\begin{smallmatrix} (Y/2)^{1/2} & (Y/2)^{-1/2}(X/2)\\ 0 & (Y/2)^{-1/2} \end{smallmatrix}\right) $ and $w=\left(\begin{smallmatrix} 0 &1 \\ -1 & 0 \end{smallmatrix}\right)$ we have $\theta(g_{Z/2}, \phi)=\theta(wg_{Z/2}, w\phi)$. On the LHS we have $\ds \theta(g_{Z/2}, \phi)=i(-1)^r(Y/2)^{1/4} \theta_r(Z).$ We compute $\theta(wg_{Z/2}, w\phi)$ below:

\begin{itemize}
	\item At $\infty$, we have $\ds r(w g_{Z/2})\phi_{\infty}(x)=\gamma_{\infty}(Y/2)^{1/2} \int\limits_{\RR} e^{2\pi i x'^2Z/2} e^{2\pi i (2xx')} dx'=
	\gamma_{\infty}(Y/2)^{1/2}  e^{4\pi i x^2 (-1/Z)} \frac{1}{\sqrt{-iZ}}.$

	\item At $p|D$, we have $\ds r(w\phi_p)=
	 \gamma_p\int\limits_{\ZZ_p+r/D}e^{-2\pi i \Frac_p(2xy)} dy 
	 =
	 \gamma_p e^{-2\pi i \Frac_p(2xr/D) }\Char_{\ZZ_p}(x)$.

	\item At $p=3$, we have $\ds r(w\phi_3)= \gamma_3\int\limits_{\ZZ_3+1/3}e^{-2\pi i \Frac_3(2xy)} dy 
	  =
	  \gamma_3e^{-2\pi i \Frac_3(2x/3)}\Char_{\ZZ_3}(x)$.

	\item At $p=2$, we have $\ds r(w\phi_2)=
	\gamma_2\int\limits_{\ZZ_2-1/2}e^{\pi i \Frac_2(y)} e^{-2\pi i \Frac_2(2xy)} dy=
	\gamma_2e^{-\pi i/2} e^{2\pi i \Frac_2(x)}\Char_{\frac{1}{2}(\ZZ_2+1/2)}(x).
	$ 

\end{itemize}

Writing all these together, we get:
\[
\theta(wg_{Z/2}, w\phi)=
\omega \frac{(Y/2)^{1/4}}{\sqrt{-iZ}}\sum_{n\in \ZZ}  e^{\pi i (n+1/2)^2 (-1/Z)} (-1)^n e^{-2\pi i \Frac_3(n/3)}e^{2\pi i \Frac_D((n+1/2)r/D) }.
\]
Changing $n \ra -n-\frac{D-1}{2}$ we get:

\[
e^{2\pi i \frac{D-1}{3}}(-1)^{\frac{D-1}{2}}\omega\frac{(Y/2)^{1/4}}{\sqrt{-iZ}}\sum_{n\in \ZZ}  e^{\pi i (n-D/2)^2 (-1/Z)} (-1)^n e^{2\pi i \Frac_3(n/3)}e^{-2\pi i \Frac_D(nr/D) }.
\]

We take the separate sums depending on $n \mod 3$.
\begin{itemize}
	\item $\ds \sum_{n\in \ZZ}  e^{\pi i (3n-D/2)^2 (-1/Z)} (-1)^n e^{-2\pi i \Frac_p(3nr/D)}=\theta^{(-3r), 1/6}(9(-1/Z))$

	\item $\ds -\omega\sum_{n\in \ZZ}  e^{\pi i (3n+D/2)^2 (-1/Z)} (-1)^n e^{2\pi i \Frac_D((3nr/D) }
	=
	-\omega\theta^{(3r), 1/6}(9(-1/Z))$.

	\item $-\omega^2\ds \sum_{n\in \ZZ}  e^{\pi i (3n-3D/2)^2 (-1/Z)} (-1)^n e^{-2\pi i \Frac_D(3nr/D) }
	=
	-\omega^2\theta^{(-3r), 1/2}(9(-1/Z)).$
	
	Thus we got:	
	\[
e^{2\pi i \frac{D-1}{3}}(-1)^{\frac{D-1}{2}}\omega\frac{(Y/2)^{1/4}}{\sqrt{-iZ}}(\theta^{(-3r), 1/6}(9(-1/Z))-\omega \theta^{(3r), 1/6}(9(-1/Z))-\omega^2\theta^{(-3r), 1/2}(9(-1/Z))).
	\]
	Taking $Z=3z$ we get the result of the lemma.

\end{itemize}

\begin{small}

\end{small}

\end{document}